\documentclass[12pt]{ejm}
\usepackage{graphicx}% Include figure files
\usepackage{dcolumn}% Align table columns on decimal point
\usepackage{bm}% bold math

\newcommand{\R}{\mathbb{R}}
\newcommand{\n}{\hat{\bm{n}}}

\usepackage{amsmath}
\usepackage{amssymb}
\usepackage{amsfonts}
\usepackage{sidecap}
\usepackage[numbers]{natbib}
\usepackage{url}

\usepackage{color}
\renewcommand{\harvardurl}[1]{\textbf{URL:} \url{#1}}

\newcommand{\bl}[1]{{\color{black} #1}}
\newcommand{\revonealso}[1]{{\color{black} #1}}
\newcommand{\ymp}[1]{{\color{black}{#1}}}
\newcommand{\revone}[1]{{\color{black}{#1}}}

\newcommand{\ympe}[1]{{\color{black}{#1}}}
\newcommand{\pt}[1]{{\color{black} #1}}

\newcommand{\yp}[1]{{\color{black}{#1}}}

\usepackage{todonotes}
\usepackage{cleveref}
\usepackage{subfig}
\usepackage{enumitem}
\newtheorem{theorem}{Theorem}[section]
\newtheorem{corollary}[theorem]{Corollary}

\newtheorem{remark}[theorem]{Remark}
\newlength{\figwidth}
\figwidth=\linewidth

\newcommand{\qed}{\nobreak \ifvmode \relax \else
      \ifdim\lastskip<1.5em \hskip-\lastskip
      \hskip1.5em plus0em minus0.5em \fi \nobreak
      \vrule height0.75em width0.5em depth0.25em\fi}      
\newcommand{\bdypt}{\bm{p}}
\newcommand{\ve}{\varepsilon}
\newcommand{\x}{\mathbf{x}}
\newcommand{\y}{\mathbf{y}}

\def\namedlabel#1#2{\begingroup
    #2%
    \def\@currentlabel{#2}%
    \phantomsection\label{#1}\endgroup
}

\title[The iPRC of Oscillators in Piecewise Smooth Systems]{The Infinitesimal Phase Response Curves of Oscillators in Piecewise Smooth Dynamical Systems}

\author[Y. Park et al.]{%
  Y.\ns P\ls A\ls R\ls K$\,^{1,3}$,\ns
  K.\ns M.\ns S\ls H\ls A\ls W$\,^{2,4}$\ns
  H.\ns J.\ns C\ls H\ls I\ls E\ls L$\,^2$\ns
\and
  P.\ns J.\ns T\ls H\ls O\ls M\ls A\ls S$\,^{1}$
}

\affiliation{%
  $^1\,$Department of Mathematics, Applied Mathematics, and Statistics. Case Western Reserve University, Cleveland, Ohio, 44106, USA.\\
  $^2\,$Department of Biology Case Western Reserve University Cleveland, Ohio, 44106, USA.\\
  $^3\,$Present address: Department of Mathematics, University of Pittsburgh, Pittsburgh, PA, 15260, USA.\\
  $^4\,$Present address: Department of anesthesia, critical care, small and pain medicine, Massachusetts General Hospital, Boston, MA 02114\\
   emails\textup{: \texttt{yop6@pitt.edu,kmshaw@partners.org,hjc@case.edu,pjthomas@case.edu}}\\}

\begin{document}

% \listoftodos

%\label{firstpage}
\maketitle
%\author{Youngmin Park\thanks{Department of Mathematics, Applied Mathematics, and Statistics. Case Western Reserve University, Cleveland, Ohio, 44106, USA.}}

%\author{\small Youngmin Park\\
%\small Department of Mathematics, Applied Mathematics, and Statistics.\\
%\small Case Western Reserve University,
%Cleveland, Ohio, 44106, USA.\\
%\small Present address: University of Pittsburgh\\
%\small \texttt{YOP6@pitt.edu}\\
%\small Kendrick M.~Shaw\\
%\small Department of Biology\\
%\small Case Western Reserve University
%\small Cleveland, Ohio, 44106, USA.\\
%\small Present address: Department of anesthesia, critical care, \\
%\small and pain medicine, Massachusetts General Hospital, Boston, MA 02114\\
%\small \texttt{kmshaw@partners.org}\\
%\small Hillel J.~Chiel\\
%\small Department of Biology\\
%\small Case Western Reserve University
%\small Cleveland, Ohio, 44106, USA\\
%\small \texttt{hjc@case.edu}\\
%\small Peter J.~Thomas\\
%\small Department of Mathematics, Applied Mathematics, and \small Statistics.\\
%\small Case Western Reserve University,
%\small Cleveland, Ohio, 44106, USA\\
%\small \texttt{pjthomas@case.edu}}

\begin{abstract}
\revonealso{The asymptotic phase $\theta$ of an initial \pt{point} $x$ in the stable manifold of a limit cycle identifies the phase of the point on the limit cycle to which the flow $\phi_t(x)$ converges as $t\to\infty$.  The infinitesimal phase response curve (iPRC) quantifies the change in timing due to a small perturbation of a limit cycle trajectory.
For a stable limit cycle in a smooth dynamical system the iPRC is the gradient $\nabla_x(\theta)$ of the phase function, which can be obtained \textit{via} the adjoint of the variational equation.  For systems with discontinuous dynamics, the standard approach to obtaining the iPRC fails.}
We derive a formula for the infinitesimal phase response curves (iPRCs) of limit cycles occurring in piecewise smooth \revonealso{(Filippov)} dynamical systems \revonealso{of arbitrary dimension}, subject to a transverse flow condition. Discontinuous jumps in the iPRC can occur at the boundaries separating subdomains\revonealso{, and are captured by a linear matching condition. The matching matrix, $M$, \revonealso{can be derived from the saltation matrix arising in the associated variational problem}.  For the special case of linear dynamics away from switching boundaries,} we obtain an explicit expression for the iPRC.  We present examples from cell biology (Glass networks) and neuroscience (central pattern generator models).  We apply the iPRCs obtained to study synchronization and phase-locking in piecewise smooth limit cycle systems \pt{in which synchronization arises solely due to the crossing of switching manifolds.}   \end{abstract}


\section{Introduction}

\subsection{Overview}

A stable limit cycle is a closed, isolated periodic orbit in a nonlinear dynamical system that attracts nearby trajectories  \cite{Guckenheimer+Holmes1990}.  Limit cycles arise in models of biological motor control systems \cite{Ijspeert:2008:NeuralNet,KelsoHoltRubinKugler1981JMotorBehav}, excitable membranes  \cite{Ermentrout1996NeuralComput,Izhikevich2007},  sensory systems \cite{FruthJuelicherLindner2014BPJ}, neuropathologies such as Parkinsonian tremor \cite{ModoloHenryBeuter2008JBiolPhys} and epilepsy \cite{SoFrancisNetoffGluckmanSchiff1998BPJ}.
%\endnote{See the intro to \cite{LinWedgwoodCoombesYoung2012JMathBiol} which has a nice selection of references also.}   
Chemical oscillations arise when the differential equations describing mass action kinetics admit a limit cycle \cite{FieldNoyes1974JChemPhys}.
Limit cycle dynamics appear not only in biological but also in engineered systems.  For instance, phase locked loops play a role in radio and electronic communications devices \cite{Stensby1997PLL_book}, and control of oscillations is an important problem in mechanical and electrical engineering \cite{RohdenSorgeTimmeWitthaut2012PRL}. 

Formally, a nonlinear autonomous $n$-dimensional ordinary differential equation,
\begin{equation}\label{eq:ode}
\frac{d}{dt}\revone{\bm{x}}(t) = \revone{\bm{F}}(\revone{\bm{x}}(t)),
\end{equation}
has a \revonealso{stable} $T$-periodic limit cycle if it admits a periodic solution $\revone{\bm{\gamma}}$ with minimal period 
\begin{equation}\label{eq:limit-cycle}
  \revone{\bm{\gamma}}(t) = \revone{\bm{\gamma}}(t+T),\quad \forall t\in \mathbb{R},
\end{equation}
and an open neighborhood of $\revone{\bm{\gamma}}$ (the basin of attraction, B.A.) within which all initial conditions give solutions converging, as $t \rightarrow \infty$, to the set
  \begin{equation}\label{eq:big-gamma}
  \Gamma = \{ \revone{\bm{\gamma}}(s) : s \in [0,T) \}.
  \end{equation}

In many situations the multidimensional dynamics of a stable limit cycle oscillator can be accurately captured in a one dimensional phase model, representing the fraction of progress around the limit cycle.  The effect of weak inputs on the oscillator can be represented in terms of their effect on the timing of the limit cycle alone, the linear approximation to which is the \emph{infinitesimal phase response curve} (iPRC).  The iPRC has become a fundamental tool for understanding entrainment and synchronization phenomena in weakly driven and weakly coupled oscillator systems, respectively \cite{ErmentroutTerman2010book,SchwemmerLewis2012PRCchapter}. The iPRC gives the relative shift in timing per unit stimulus, as a function of the phase at which the stimulus occurs, in the limit of small stimulus size (Figure \ref{fig:1d_example}\yp{B}; \cite{ErmentroutTerman2010book}). 
\revonealso{The iPRC $\bm{z}(t)$ is a vector quantity; it is equivalent to the gradient of the asymptotic phase function $\theta:\text{B.A.}\to[0,1)$. \yp{The phase function $\theta$ maps each point in the basin of attraction to a point on the circle labeling the limit cycle trajectory to which it converges as $t\to\infty$. In particular, the labeling is simply chosen to be $\theta = t/T \mod T$.} A displacement from the limit cycle by amount $\Delta\bm{y}$ at time $t\in[0,T)$ causes a shift in timing equal to $T\Delta\bm{y}\cdot \bm{z}(t) + o(|\Delta\bm{y}|)$, in the limit as $|\Delta\bm{y}|\to 0$.}

The iPRC is known in closed form in a handful of special cases: near a supercritical Andronov-Hopf bifurcation \cite{Izhikevich2007}, near a saddle-node-on-invariant-circle (SNIC) bifurcation \cite{BrownMoehlisHolmes2004NeComp}, and for certain piecewise linear oscillator models \cite{Coombes:2008:SIADS,CoombesThulWedgwood2012PhysD,ShawParkChielThomas2012SIADS}.  The  form of the iPRC near a  homoclinic bifurcation is not known in general, cf.~\cite{BrownMoehlisHolmes2004NeComp,LinWedgwoodCoombesYoung2012JMathBiol,ShawParkChielThomas2012SIADS}.

%%% aplysia/papers/yxp30/j-math-neuro/glass\_pert\_displacement\_fig.pdf

For general smooth nonlinear systems with limit cycle dynamics, one may obtain the iPRC numerically using an adjoint method \cite{IzhikevichErmentrout2008Scholarpedia} \revonealso{or via continuation of a two-point boundary value problem \cite{OsingaMoehlis2010SIADS}}.  If oscillations arise from a dynamical system \eqref{eq:ode} where $\revone{\bm{F}}:\R^n\to\R^n$ is a $C^1$ differentiable map,  then the iPRC is a $T$-periodic vector function of time that obeys an adjoint equation 
\begin{equation}
d\revone{\bm{z}}/dt=A(t)\revone{\bm{z}}(t),\quad A(t)=-\left(DF(\revone{\bm{\gamma}}(t))\right)^\intercal
\end{equation}
together with the boundary condition $\revone{\bm{z}}(t)=\revone{\bm{z}}(t+T)$.  The ``adjoint operator" $A(t)$ is the transpose of the \yp{negative} Jacobian matrix $DF$ evaluated at the limit cycle $\revone{\bm{\gamma}}(t)$ \cite{ErmentroutTerman2010book,Izhikevich2007}.  

\revonealso{Both the adjoint equation method and the continuation-based method break down} for nonsmooth systems, i.e.~for systems such that the Jacobian $DF$ is not defined at all points around the limit cycle. \revonealso{Such cases arise in piecewise smooth systems, where the vector field changes abruptly across some boundary $\Sigma$. The monograph  \cite{BernardoBuddChampneysKowalczyk2008PiecewiseSmoothDynSysBook} classifies piecewise smooth systems according to the degree of smoothness at the boundaries. If across a boundary $\Sigma$, the vector fields $F_1,F_2$ are discontinuous ($F_1(\x) \neq F_2(\x)$) at a point $\x \in \Sigma$, then the system is said to have degree of smoothness one, and are said to be of \textit{Filippov} type (also called differential inclusions \cite{filipov1988}; the derivative of a solution passing through such a point may take a value in some well-defined set, rather than equalling a unique value). If the vector fields across the boundary satisfy $F_1(\x) = F_2(\x)$, but differ in their Jacobians ($DF_1 \neq DF_2$) at $\x$, then the degree of smoothness is said to be 2. This definition generalizes to higher order derivatives in a natural way. Systems with smoothness degree two or higher are called \textit{piecewise-smooth continuous systems} (\cite{BernardoBuddChampneysKowalczyk2008PiecewiseSmoothDynSysBook}, p.~74). Finally, a system with a discontinuous solution at the boundary $\Sigma$ is said to have degree of smoothness zero. In this paper we consider systems with smoothness one or higher, i.e.~we assume the solutions are continuous functions of time.}

\revonealso{If $F$ is either a Filippov system or piecewise smooth continuous system}, the Jacobian linearization may break down at the boundaries separating the regions within which the vector field is smooth. The theory of isochrons, and the subsequent iPRC, is well developed for smooth systems \cite{Guckenheimer1975JMathBiol}, with considerably less literature for iPRCs in nonsmooth differential equations. Infinitesimal PRCs have been computed explicitly in some planar systems \cite{Coombes:2008:SIADS,ShawParkChielThomas2012SIADS}. 
\pt{Recent literature shows a significant interest in the analysis of iPRCs for piecewise smooth systems in both biological and control engineering contexts \cite{coombes2001phase,CoombesThulWedgwood2012PhysD,izhikevich2000phase,shirasaka2017phase}. We compare these studies in detail in \S \ref{ssec:related_literature}.}

Nonsmooth oscillator models arise in both biological and engineered systems. %\footnote{Not to mention physical systems!  PRE is a physics journal after all.  Try google search on ``nonsmooth oscillator physics limit cycle'' (google and also google scholar?). And references in \cite{KuznetsovRinaldiGragnani2003IJBC}.}  
Examples include planar nonlinear integrate-and-fire neural models \cite{CoombesThulWedgwood2012PhysD},  piecewise linear approximations to the Hindmarsch-Rose neural model \cite{poggi_etal_2009}, and models of anti-lock braking systems \cite{morse_1997,pettit_wellstead_1995}. 
Section \S \ref{ssec:further_applications}  mentions additional examples.  
Existence of oscillatory solutions in piecewise smooth systems is a question of interest in its own right \cite{filipov1988,huan_etal_2012,LlibrePonce2012DCDISSBAA,morse_1997,gaiko_van_horssen_2009}.
    
In this paper we \pt{derive} a formula for the infinitesimal phase response curve of stable limit cycles that allows for  discontinuities at the domain boundaries of piecewise smooth dynamical systems. In the case of piecewise \emph{linear} systems, we obtain an explicit expression in terms of the system coefficients for each subdomain through which the limit cycle travels, and the tangent vectors of the surfaces separating the regions where the vector field definition changes (equation \eqref{eq:iprc_exact} in \S \ref{sec:theorem}).  To obtain these results we derive a jump condition satisfied by the iPRC at the boundaries between subdomains.   

\revonealso{\textbf{Overview of the paper:} In the next section \S 1.2 we develop a motivating one-dimensional example in detail.  In \S 2 we calculate the form of the discontinuity in the infinitesimal phase response curve for a limit cycle in a piecewise smooth dynamical system in arbitrary dimensions. In the case of limit cycles arising in $n$-dimensional piecewise \emph{linear} dynamical systems, we provide an explicit closed form for the iPRC. \S 2.1 lays out the assumptions needed to establish our results, and \S 2.2 presents the main Theorem (2.1) giving the correction to the iPRC upon crossing a switching boundary. \S 3 provides examples of iPRCs in nonsmooth systems: \yp{a planar piecewise constant system where the nonlinearities arise strictly from the boundaries in \S 3.1, a planar piecewise linear oscillator introduced in a motor control context \cite{ShawParkChielThomas2012SIADS}, but generalized here to a non-symmetric geometry in \S 3.2, a piecewise linear genetic regulatory circuit model (Glass network \cite{GlassPasternack1978JMB,GlassPerez1974JChemPhys}) in \S 3.3, a three-dimensional motor control model \cite{ShawLyttleGillCullinsMcManusLuThomasChiel2015JCNS} in \S 3.4, a four-dimensional weakly diffusively coupled version of the piecewise constant system in \S 3.4.1, and a six-dimensional threshold linear network model comprising of two weakly coupled three-dimensional oscillators \cite{MorrisonDegeratuItskovCurto2016arXiv} in \S 3.4.2.}  In \S 4.1 we discuss the relation between our boundary-crossing correction matrix and the classical saltation matrix, in \S 4.2 we discuss the limitations of the method, and  in \S 4.3 we discuss a range of possible further applications.  Following our conclusion  \S 5, the appendices detail the proofs and derivations of the  results.}

The main results reported here appeared previously in the Master's thesis of the first author \cite{park2013infinitesimal}.

\begin{figure}[htbp] %  figure placement: here, top, bottom, or page
   \centering
   \includegraphics[width=\textwidth]{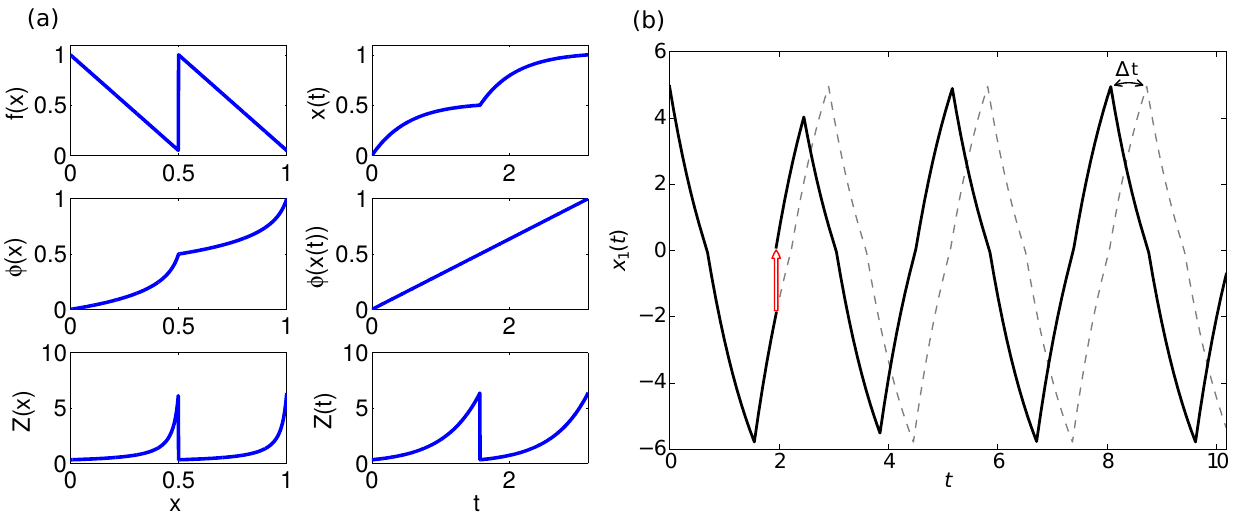} 
   \caption{\bl{\yp{1D oscillator with piecewise smooth velocity. \textbf{A}. \textbf{Top left:} Velocity $f(x)=1-2\alpha x$ for $0\le x < 1/2$ and $f(x)=1-2\alpha(x-1/2)$ for $1/2\le x < 1$, with $\alpha=0.95$. The  period is $T=\frac{1}{\alpha}\ln\left(\frac{1}{1-\alpha}\right)\approx 3.153$. \textbf{Middle right:} trajectory $x(t)$. \textbf{Center:} oscillator phase $\phi$ plotted versus space ($x$, \textbf{left}) and versus time ($t$, \textbf{right}).  \textbf{Bottom:} infinitesimal phase response curve $Z$ plotted versus space ($x$, \textbf{left}) and versus time ($t$, \textbf{right}).  The phase $\phi$ is a continuous mapping from $x$ to the circle $[0,1)$. The iPRC jumps by $-\frac{\alpha^2}{1-\alpha}\left[ \ln\frac{1}{1-\alpha}\right]^{-1}\approx-6.025$ at $x=1/2$ and again at  $x=0$. \textbf{B}. Direct perturbation (red arrow) and phase response ($\Delta \theta$) for a limit cycle solution of a 2-D Glass network model. For a perturbation of size $\varepsilon$, we recover the iPRC value  as  $\lim_{\varepsilon \rightarrow 0} \Delta\theta/\varepsilon$ (\cite{ErmentroutTerman2010book}). The $\Delta t$ in the figure is equivalent to $\Delta \theta T$, where $T$ is the oscillator period.}}}
   \label{fig:1d_example}
\end{figure}

\subsection{A 1D example} To illustrate the necessity of  including a jump condition in the phase response for piecewise smooth systems, consider the following one-dimensional example \pt{(see also \cite{coombes2001phase,izhikevich2000phase})}.  

Let $f_1$ and $f_2$ be smooth, strictly positive functions defined on the unit interval.  Identify the  interval with the circle and let $x\in[0,1)$ evolve according to
\begin{equation}
\frac{dx}{dt}=\left\{\begin{array}{rr}f_1(x),&0\le x <a\\f_2(x),&a\le x <1,
\end{array}
\right.
\end{equation}
where $0<a<1$ marks the location at which the rate law for $x$ changes from $f_1$ to $f_2$.  The rate law changes back to $f_1$ when $x$ wraps around from one to zero again.  The period of this oscillator is 
$$T=\int_0^a\frac{dx}{f_1(x)}+\int_a^1\frac{dx}{f_2(x)}$$
We can define a phase variable $\phi(x)$ by the condition $d\phi/dt=1/T$, which gives the form
$$\phi(x)-\phi(0)=\left\{\begin{array}{rr}
\frac1T\int_0^x\frac{\revonealso{dx'}}{f_1(\revonealso{x'})},&0\le x \le a\\
\phi(a)+\frac1T\int_a^x\frac{\revonealso{dx'}}{f_2(\revonealso{x'})},&a<x\le 1
\end{array}
\right.
$$
Here $\phi(0)$ is an arbitrary constant which we may set to zero, without loss of generality.  
The infinitesimal phase response curve $Z_x$ for this system describes the shift in timing of the oscillation upon making a small displacement in the $x$ coordinate, as a function of position.  The iPRC is 
\begin{equation}
Z_x=\frac{d\phi}{dx}=\left\{  
\begin{array}{rr}
Z_1\equiv(Tf_1(x))^{-1},&0\le x < a\\
Z_2\equiv(Tf_2(x))^{-1},&a<x\le 1.
\end{array}
\right.
\end{equation}
The iPRC has a finite jump discontinuity at the location $a$ where the rate law for $x$ changes, namely
\begin{equation}
Z_2(a)-Z_1(a)=\frac1T\left( \frac1{f_2(a)}-\frac1{f_1(a)}  \right).
\end{equation}
\bl{As a specific example, consider the rate laws  $f_1(x)=1-2\alpha x$, $f_2(x)=1-2\alpha(x-1/2)$, parameterized by $\alpha<1$, with switching point  $a=1/2$.
For this example $T=\frac{1}{\alpha}\ln\left(\frac{1}{1-\alpha}\right)$, $\phi(x)=\frac{-1}{2\ln (1-\alpha)}\ln\left(\frac1{1-2\alpha x}\right)$ for $0\le x \le 1/2$, and $\phi(x)=\frac12 -\frac{1}{2\ln (1-\alpha)}\ln\left(\frac1{1-2\alpha (x-1/2)}\right)$ for $1/2\le x < 1$.  The phase response curves in the first and second intervals are $Z_1(t)=\frac\alpha{1-2\alpha x(t)}\left[\ln\frac1{1-\alpha}\right]^{-1}$ and $Z_2(t)=\frac\alpha{1-2\alpha (x(t)-1/2)}\left[\ln\frac1{1-\alpha}\right]^{-1}$, respectively.  The phase $\phi(x)$ is continuous across the switch points $x=1/2$ and $x=0$. The jump in the infinitesimal phase response curve is $Z_2(T_a)-Z_1(T_a)=-\frac{\alpha^2}{1-\alpha}\left[ \ln\frac{1}{1-\alpha}\right]^{-1}$; here $T_a=T/2$ is the time at which the trajectory reaches the switching point $a=1/2$.  Figure \ref{fig:1d_example} illustrates this scenario for $\alpha=0.95$.  }

\yp{We remark that the size of the discontinuity in this example could be computed explicitly because the domain is one-dimensional, and the normalization condition $d\phi/dt = 1/T$ yields one equation with one unknown for the size of the discontinuity. However, in $n > 1$ dimensions, the normalization condition only yields one equation in $n$ unknowns. In the following section, we derive the remaining $n-1$ equations required to determine the size of the discontinuity in the general case.}

\section{Methods}
\subsection{Definitions and Hypotheses Required to Solve the Adjoint Equation Over Differential Inclusions}\label{section:definitions-and-hypotheses}

We introduce notation needed to discuss infinitesimal phase  response curves for differential inclusion systems.  For a general treatment of Filippov systems see \cite{filipov1988}.
  
\emph{Limit cycles and asymptotic phase.}  In a smooth system of the form \eqref{eq:ode}, possessing a stable limit cycle $\revone{\bm{\gamma}}$, we associate a phase $\theta \in [0,1)$ with points along the  cycle such that
\begin{equation}\label{eq:phase}
  d\theta(\revone{\bm{\gamma}}(t))/dt = 1/T,
\end{equation}
%
% This figure is not referred to in the text -PJT 2016-03-07
%\begin{figure}[ht]
%\includegraphics[width=\linewidth]{glass_2d_phase_fig.pdf}
% \caption{Each point on the limit cycle is assigned a phase value.  The top panel shows each state variable of the limit cycle trajectory, $x_1$ (black) and $x_2$ (gray) as a function of time.  The period of the limit cycle is shown in the same panel as $T \approx 2.9$.  The lower panel plots the phase, which evolves at a constant rate, \textit{cf.}~Eq.~\eqref{eq:phase}.}
%\end{figure}
%\todo{remove vertical lines from glass\_2d\_phase\_fig.pdf}
where $T$ is the period of the limit cycle and $\theta(\revone{\bm{\gamma}}(t_0)) = 0$ is chosen arbitrarily. To each point $\revone{\bm{x}}_0$ in the basin of attraction (B.A.) we assign a phase $\theta(\revone{\bm{x}}_0)\in[0,1)$ such that the trajectory $\revone{\bm{x}}(t)$ with initial condition $\revone{\bm{x}}(0)=\revone{\bm{x}}_0$ satisfies
\begin{equation}\label{eq:thetax0}
    \lim_{t \rightarrow \infty} \| \revone{\bm{x}}(t) - \revone{\bm{\gamma}}(t + T\theta(\revone{\bm{x}}_0)) \| \rightarrow 0.
  \end{equation}
The isochrons are level curves of the phase function $\theta(\revone{\bm{x}}_0)$, and foliate the basin of attraction. For a stable limit cycle in a smooth dynamical system, the existence of the phase function is a well known consequence of results from invariant manifold theory \cite{Guckenheimer1975JMathBiol}.
Intuitively, isochrons indicate which points in the basin of attraction eventually converge to the limit cycle solution having a particular phase.

\emph{Filippov systems.}  Let $D$ be a path connected subset of $ \mathbb{R}^n$.  We say that an autonomous vector field $\bm{F}:D \rightarrow \mathbb{R}^n$ is piecewise smooth on $D$ if there exist a finite number, $R$, of open sets $D_r$ such that the following hypotheses hold:
  \begin{enumerate}[label=\arabic*.]
    \item[H1.] \label{h1}$D_r$ is nonempty, simply connected, and open for each $r$.
    \item[H2.] \label{h2}$D_i \cap D_j = \varnothing, \forall i \neq j$.
    \item[H3.] \label{h3}$D \subset \bigcup_{r=1}^R \bar{D}_r$.
    \item[H4.] \label{h4}There exist \revonealso{$C^1$}, bounded vector fields $\bm{F}_r: \bar{D}_r \rightarrow \mathbb{R}^n$ such that for all $x$ in $D_r$, $\bm{F}_r(x)=\bm{F}(x)$.
  \end{enumerate}
  Note that we require $\bm{F}_r=\bm{F}$ only on the interior of each open domain $D_r$, while we require that each $\bm{F}_r$ be smooth on the closure $\bar{D}_r$.
  
The corresponding dynamical system 
\begin{equation}\label{eq:Filipovsystem}
\frac{d\revone{\bm{x}}}{dt}=\bm{F}(\revone{\bm{x}})
\end{equation}
is called a piecewise smooth dynamical system or a Filippov system \cite{filipov1988}.

\begin{figure*}[htbp] %  figure placement: here, top, bottom, or page
   \centering
   \includegraphics[width=4in]{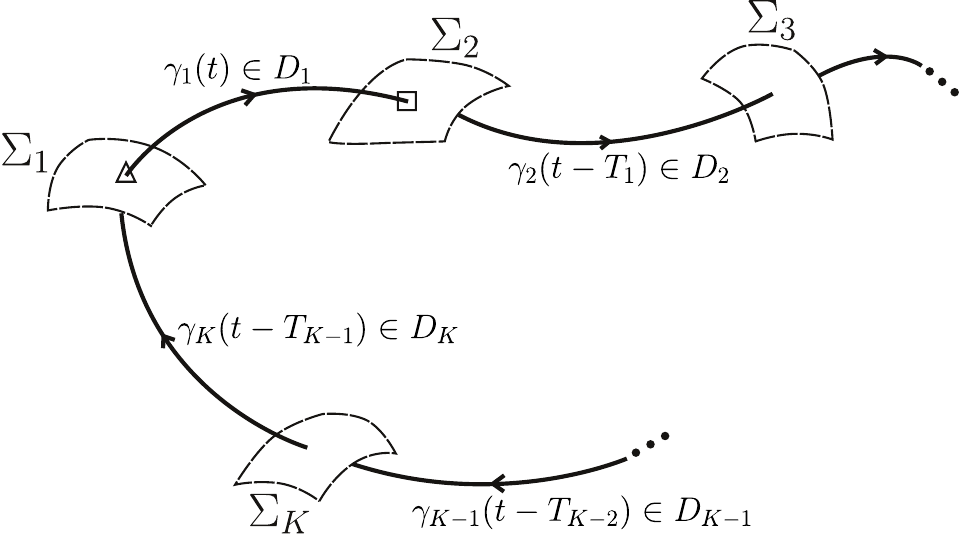} 
   \caption{Sections of the piecewise smooth dynamical system. Full domain boundaries are omitted for clarity.   The first segment of the limit cycle begins at point $\revone{\bm{\gamma}}_1(0)\equiv\bdypt_{1}$ (marked $\triangle$) on the surface $\Sigma_{1}$.   At time $T_1=t_1$ the trajectory crosses surface $\Sigma_{2}$ at point $\revone{\bm{\gamma}}_1(t_1)\equiv\bdypt_{2}\equiv\revone{\bm{\gamma}}_2(0)$ (marked $\square$).   The $j$th segment of the limit cycle travels from $\revone{\bm{\gamma}}_j(0)\in\Sigma_{j}$ to $\revone{\bm{\gamma}}_j(t_j)\in\Sigma_{j+1}$ in time $t_j$;    the crossing from region $j$ to region $j+1$ occurs at \bl{global} time $T_j=t_1+t_2+\ldots+t_j$ and at location $\revone{\bm{\gamma}}_j(t_j)\equiv\bdypt_{j+1}\equiv\revone{\bm{\gamma}}_{j+1}(0)$.  The cycle returns to its starting point, $\bdypt_{1}$ (marked $\triangle$), at time $T_K=t_1+\ldots+t_K$.    Thus $T=T_K$ is the period of the limit cycle.}
   \label{fig:sections}
\end{figure*}
  We further restrict our attention to Filippov systems satisfying the following assumptions:
  \begin{enumerate}[label=\arabic*.]
  \item \label{a:limit_cycle} \textit{The system \eqref{eq:Filipovsystem} has a stable $T$-periodic limit cycle that crosses each boundary transversely with nonzero speed.}
  Because the limit cycle could cross the same boundary multiple times, we introduce a separate label for each segment of the limit cycle lying between two boundary crossings (Figure \ref{fig:sections}).  Thus, we label each piecewise smooth portion of the limit cycle by a number $k=1,\ldots,K$ (see Eq.~\eqref{eq:gamma_pieces}), \yp{and use the same convention of the regions}. \revonealso{The boundary between the $k^\text{th}$ and $(k+1)^\text{st}$ portions of the limit cycle \yp{(between regions $D_k$ and $D_{k+1}$)} is a surface denoted $\Sigma_{k+1}$}.  We denote the point at which the limit cycle crosses this boundary by $\bdypt_{k+1} \in \Sigma_{k+1}$. 
  
    \newcommand{\tvi}{\hat{\bm{w}}_{i}} % tangent vector i  
  \item \label{a:sigma} 
  \textit{Each boundary is at least $C^1$ in an \bl{open ball  $B(\bdypt_{k+1},c)$, centered at $\bdypt_{k+1}$ with radius $c$.} } From this assumption it follows that at each crossing point $\bdypt_{k+1}$ there exists a tangent hyperplane spanned by $n-1$ orthonormal basis vectors, denoted
   $\tvi^{k+1}$ for $i=1,\ldots,n-1$, and a unique normal vector $\n_{k+1}$ \bl{directed from region $k$ towards region $k+1$}. Moreover, the vector fields $\bm{F}_{k}$ and $\bm{F}_{k+1}$ are assumed to satisfy a transverse crossing condition: 
  \begin{align}
  \bm{F}_k(\bdypt_{k+1})\cdot\n_{k+1}&>0\\
  \bm{F}_{k+1}(\bdypt_{k+1})\cdot\n_{k+1}&>0.
  \end{align}
     
  \item \label{a:phase}\textit{\revonealso{The limit cycle of system \eqref{eq:Filipovsystem} admits} a phase function that can be extended to a continuous function $\theta:$B.A.$\to S^1$ from the basin of attraction to the circle $S^1\equiv [0,1],$ satisfying 
  $$\frac{d}{dt}\theta(\revone{\bm{x}}(t))=\frac1T$$ along trajectories within the basin of attraction.}%  See Eqs.~\eqref{eq:limit-cycle} and ~\eqref{eq:phase} for definitions of limit cycle and phase.
%    
%    Suppose that $\gamma(t)$ is a T-periodic limit cycle solution satisfying a particular differential inclusion.  We adopt the definition of phase from smooth systems; it is a function $\theta: \gamma \rightarrow S^1 = [0,1)$ such that
%    \begin{equation}
%     \frac{d}{dt}\theta(\gamma(t)) = 1/T,
%    \end{equation}
%    where $\theta=0$ is chosen arbitrarily.

%   \item \label{a4} \textbf{REMOVE -- combined with \# 3.  But check references to assumption numbers.}  \textit{A unique asymptotic phase, $\theta$, exists for each point, $x_0$, in the basin of attraction (B.A.), and
%   \begin{equation}
%     \frac{d}{dt}\theta(x(t)) = 1/T, \quad \forall x(t) \in B.A.,
%   \end{equation} where $x(0) = x_0$.
%   }
  \item \label{a:isochron} \textit{The level sets of the phase function $\theta$ (the isochronal surfaces) form a continuous foliation of \revonealso{an open neighborhood of the limit cycle.}}
  %, i.e., for each $x \in B.A.$, \begin{equation}
% \forall \varepsilon > 0, \exists \delta >0 \rm{ s.t. } y \in B(x,\delta) \Rightarrow d(\theta(x),\theta(y)) < \varepsilon.
% \end{equation}
\item \label{a:phase_deriv}\textit{The phase function, $\theta$, is differentiable within the interior of each region for which it is defined.}

\item \label{a:direct_deriv}\textit{At each boundary crossing point $\bdypt_{k+1}\in\Sigma_{k+1}$, the directional derivative of the phase function is defined in the directions of each of the $n-1$ tangent vectors $\tvi^{k+1}$.}

\end{enumerate}

\yp{\begin{remark} Although we label conditions 1-6 as ``assumptions", it has been recently established that conditions 3-6 follow from conditions 1-2 \cite{shirasaka2017phase}.  All six conditions hold for the model systems we consider in the examples.
%, the authors define and show existence of isochrons for general hybrid systems with transverse solution crossings, showing that Assumptions \ref{a:phase}--\ref{a:direct_deriv} exist in general. The examples we consider are included in their derivations, thus we freely use Assumptions \ref{a:phase}--\ref{a:direct_deriv} in each example.
    \end{remark}
}

For smooth systems the iPRC $\revone{\bm{z}}(t)=\nabla\theta(\revone{\bm{\gamma}}(t))$ may be found using an adjoint equation,
\begin{equation}\label{eq:adjoint}
 \frac{d\revone{\bm{z}}(t)}{dt} = A(t) \revone{\bm{z}}(t),
\end{equation}
where $A(t) = -DF^T(\revone{\bm{\gamma}}(t))$, the negative transpose of the linearization of the vector field $\revone{\bm{F}}$ evaluated along the limit cycle $\revone{\bm{\gamma}}$. 
To derive \eqref{eq:adjoint} one considers an infinitesimal perturbation $\revone{\bm{x}}(t)=\revone{\bm{\gamma}}(t)+\revone{\bm{y}}(t)$ where $||\revone{\bm{y}}(t)||\ll 1$ and $\revone{\bm{\gamma}}(t)$ is the limit cycle.  As observed in \cite{BrownMoehlisHolmes2004NeComp}, $\revone{\bm{y}}(t)\cdot \revone{\bm{z}}(t)$ is independent of time; setting 
\begin{equation}\label{eq:deriveadjoint}
\frac{d}{dt}\left(\revone{\bm{y}}(t)\cdot \revone{\bm{z}}(t)\right)=0
\end{equation}
gives an operator equation %$L^*[z(t)]=0$ \footnote{Why is this $L^*$? This notation is not used anywhere else in the paper and is not clearly defined here.  Remove.}  
leading to the adjoint.  We note that \eqref{eq:deriveadjoint} holds for piecewise smooth systems \emph{within the interior of each subdomain}.  We will develop a parallel method for piecewise smooth vector fields and solve for each limit cycle section $\revone{\bm{\gamma}}_k$, that is,
  \begin{equation}\label{eq:adjoint-piecewise}
   \frac{d\bm{z}_k(t)}{dt} = A_k(t) \bm{z}_k(t),
  \end{equation}
where $A_k(t) = -DF_k(\revone{\bm{\gamma}}_k(t))^T$, the negative transpose of the linearization of the vector field $\bm{F}_k$ evaluated along the limit cycle portion $\revone{\bm{\gamma}}_k$.
The remaining challenge, and the contribution of the paper, is to establish the  conditions relating the iPRC on either side of each boundary crossing.
  
 %However, numerical simulations show discontinuities in the iPRC for piecewise linear vector fields \cite{ShawParkChielThomas2012SIADS}, but the adjoint equation, Eq.~\eqref{eq:adjoint-piecewise}, does not reveal how these discontinuities arise and how to account for them in deriving an accurate iPRC.  In the next section, we present in detail the calculations that must be done to account for these discontinuities for any system of piecewise linear differential equations satisfying the hypotheses and assumptions above.
  
\subsection{Solving the Boundary Problem of the Adjoint Equation}\label{sec:theorem}
  
We fix notation and define additional terms. Let $\bm{F}_k$ denote the vector field in which the $k$th portion of the limit cycle resides, where each $\bm{F}_k$ is numbered sequentially.
%\sout{according to the direction of flow within the ball $B(\bdypt_{k+1},c)$ centered at $\bdypt_{k+1}$ with radius $c>0$.}\footnote{PT: as defined earlier, $B(\bdypt_{k+1},c)\subset\Sigma_{k+1}$, so it is flat and has no interior.  So ``the direction of flow within the ball" is meanginless".}  \sout{Note that while each $\bm{F}_r$ is unique using the domain numbering $r$, the same is not necessarily true for each $\bm{F}_k$, using the limit cycle crossing numbering $k$.}

The limit cycle, $\revone{\bm{\gamma}}$, is piecewise smooth, consisting of several curves $\revone{\bm{\gamma}}_1, \revone{\bm{\gamma}}_2, \ldots, \revone{\bm{\gamma}}_K$. As illustrated in Figure \ref{fig:sections}, each $\revone{\bm{\gamma}}_k$ spends a time $t_k$ in some domain $D_r$.  We write the limit cycle $\revone{\bm{\gamma}}$ as a collection of curves,
\begin{eqnarray}\label{eq:gamma_pieces}
  \revone{\bm{\gamma}}(t) = \left\{
  \begin{aligned}
    &\revone{\bm{\gamma}}_1(t), \quad 0=T_0 \leq t < T_1,\\
    &\revone{\bm{\gamma}}_2(t-T_1), \quad T_1 \leq t < T_2,\\
    &\vdots\\
    &\revone{\bm{\gamma}}_K \left ( t - T_{K-1} \right ), \quad T_{K-1} \leq t < T_K,
  \end{aligned}
  \right.
\end{eqnarray}
where $T_i = \sum_{j=1}^i t_j$ is the global time at which the trajectory crosses boundary surface $\Sigma_{i+1}$, and $\revone{\bm{\gamma}}_k(t_k) = \revone{\bm{\gamma}}_{k+1}(0)$ enforces continuity of the limit cycle. At a limit cycle boundary crossing \revonealso{$\Sigma_{k+1}$} between the $k$th and $(k+1)$st portions of the limit cycle, there exist two adjacent vector fields $\bm{F}_{k}$ and $\bm{F}_{k+1}$.  These vector fields evaluated at the limit cycle boundary crossing are denoted
\begin{equation}\label{eq:vector-field-gamma}
\begin{split}
\bm{F}_{k,t_k}&= \lim_{t \rightarrow t_k^-} \bm{F}_k(\revone{\bm{\gamma}}_{k}(t)),\\
\bm{F}_{k+1,0}&= \lim_{t \rightarrow 0^+} \bm{F}_{k+1}(\revone{\bm{\gamma}}_{k+1}(t)). 
\end{split}
\end{equation}
In Eq.~\eqref{eq:vector-field-gamma} and for the rest of this section, \ymp{unless stated otherwise,} the value $t$ will refer to the time elapsed within a particular region between boundary crossings, \textit{i.e.}~for the $k$th limit cycle segment, $t \in [0,t_k)$.  The one-sided limits exist because each vector field is required to be smooth on the closure of its domain.  The global iPRC, $\bm{z}$, will be defined in terms of the phase variable $\theta \in [0,1)$, but we will view the local iPRC $\bm{z}_k(t)$ in terms of local time. The independent variable of these iPRCs are related by
\begin{equation}\label{eq:global-to-local-phase}
\bm{z}(\theta) = \bm{z}_k(\theta T - T_{k-1}) = \bm{z}_k(t).
\end{equation}
We will use the local time $t$ in the proofs to follow, so that we only need to consider local dynamics at an arbitrary boundary crossing, without having to refer to the global dynamics. 

We define additional terms $\bm{z}_{k,t_k}$ and $\bm{z}_{k+1,0}$  by
\begin{equation}\label{eq:iprc-limits}
\begin{split}
\bm{z}_{k,t_k}&=\lim_{t\to t_k^-}\bm{z}_k(t),\\
\bm{z}_{k+1,0}&=\lim_{t\to 0^+}\bm{z}_{k+1}(t)
\end{split}
\end{equation}
where $\bm{z}_{k}$ and $\bm{z}_{k+1}$ are the solutions to the adjoint equation (Eq.~\eqref{eq:adjoint-piecewise}) over vector fields $\bm{F}_{k}$ and $\bm{F}_{k+1}$, respectively.  As a rule, the first entry of the subscript for either $\bm{z}_{k,t}$ or $\bm{F}_{k,t}$ denotes the limit cycle section, and the second entry of the subscript (when explicit) denotes the local time.

\yp{
\begin{theorem}\label{theorem}
Consider a particular domain boundary $\Sigma$ and a piecewise smooth limit cycle $\gamma(t)$ satisfying hypotheses H1--H4 and assumptions \ref{a:limit_cycle}--\ref{a:direct_deriv} that transversely crosses $\Sigma$ at point $\mathbf{p}$ at time $t=0$, exiting the old domain with velocity $\bm{F}^-$ and entering the new domain with velocity $\bm{F}^+$.  For brevity, write $\mathbf{z}^-$ for the iPRC vector $\lim_{t\to 0^-}\mathbf{z}(t)$ just before the crossing, and $\mathbf{z}^+$ for the iPRC vector $\lim_{t\to 0^+}\mathbf{z}(t)$ immediately after the crossing.  In the interior of domain $k$, the iPRC vector evolves according to $\dot{\mathbf{z}}=-(D\mathbf{F}_k(\gamma(t)))^\intercal\mathbf{z}.$  The boundary crossing induces a linear jump condition $A\mathbf{z}^+=B\mathbf{z}^-$.
If $\bm{w}_1,\ldots,\bm{w}_{n-1}$ is any orthonormal basis for the tangent space of $\Sigma$ at $\bm{p}$, the matrices $A$ and $B$ are given 
\begin{equation}\label{eq:saltation}
\yp{C}=\left(\begin{matrix}\bm{F}^+&|& \bm{w}_1 &|& \cdots &|& \bm{w}_{n-1}\end{matrix}\right)^\intercal,\quad
\yp{D}=\left(\begin{matrix}\bm{F}^-&|& \bm{w}_1 &|& \cdots &|& \bm{w}_{n-1}\end{matrix}\right)^\intercal.
\end{equation}
\end{theorem}
For convenience, we write 
\begin{equation}\label{eq:matrix-a-b}
M_{k+1} = \yp{C}_{k+1}^{-1} \yp{D}_{k},
\end{equation}
and call $M_{k+1}$ the ``jump matrix'' for the boundary crossing from region $k$ to $k+1$. Existence of the required matrix inverse is guaranteed by the transverse flow condition.

\textbf{Proof outline}: Requiring that the directional derivatives match at the boundary results in $n-1$ equations with $n$ unknowns. The final equation is determined by matching the normalization condition $F\cdot z = 1/T$ on both sides of the boundary. The resulting system is linear, and can be expressed in the matrix form above. See \S \ref{section:proof-of-theorem} for the proof of Theorem \ref{theorem}.
}

The following two corollaries specialize to the case of piecewise linear vector fields.  In this case the vector field is not only piecewise smooth, but the iPRC may be obtained in terms of matrix exponentials.

\begin{corollary}\label{corl:circuit}
With the assumptions of Theorem \ref{theorem} and affine linear vector fields $\bm{F}_k$, the initial condition of the iPRC, $\revone{\bm{z}}(\theta)$, must satisfy
\begin{equation}\label{eq:eigenvalue}
\bm{z}_{1,0}=B \bm{z}_{1,0},
\end{equation}
where
\begin{equation}
B=M_{1}e^{A_{K}t_{K}}M_K\cdots M_{k+1}e^{A_{k}t_{k}}M_k \cdots M_{3}e^{A_{2} t_{2}}M_{2}e^{A_{1} t_{1}},
\end{equation}
$t_k$ is the time of flight of the $k$th portion of the limit cycle, $e^{A_{k}t_k}$ is the matrix exponential solution of the adjoint equation with  $A_k = -\left(DF_k\right)^\intercal$ at time $t_k$, and $DF_k$ denotes the Jacobian matrix of the vector field $\bm{F}_k$.  Eq.~\eqref{eq:eigenvalue} and the normalization condition,
\revonealso{\begin{equation}\label{eq:normalization}
\bm{F}_{1,0}\cdot \bm{z}_{1,0}=\frac{1}{T},
\end{equation}}
yield a unique solution for the initial condition, $\bm{z}_{1,0} \in \mathbb{R}^n$. 
\end{corollary}
\yp{\textbf{Proof outline}: By iterating the adjoint solution and jump matrices forward in time, we must return to the same initial condition, which results in Equation \ref{eq:eigenvalue}. This equation is an eigenvalue problem, and reveals that the initial condition of the iPRC is the unit eigenvector of the matrix $B$ up to scalar multiplication. The scalar multiple is uniquely determined by the normalization condition $F\cdot z = 1/T$.  See \S \ref{section:proof-of-corollaryIII2} for the proof of Corollary \ref{corl:circuit}}

% \bl{\begin{remark}
%  The Jacobian matrices $A_{k}$ in Corollary \ref{corl:circuit} are distinct from the matrices $\yp{C}_{k}$ appearing in the jump condition in Theorem \ref{theorem}.
% \end{remark}}

\begin{corollary}\label{corl:eqn}
\ymp{Let $t$ denote global time.} Under the assumptions of Corollary \ref{corl:circuit}, the iPRC is given by
\begin{equation}
 \bm{z}(t) = \begin{cases} 
      e^{A_1t}\bm{z}_{1,0} \equiv e^{A_1 t} B\bm{z}_{1,0} & 0 \leq t < T_1 \\
      e^{A_2\ymp{(t-T_{1})}}B_1\bm{z}_{1,0} & T_1 \leq t < T_2 \\
      \vdots\\
      e^{A_K\ymp{(t-T_{K-1})}}B_{K-1}\bm{z}_{1,0} & T_{K-1} \leq t < T_K
   \end{cases}
\end{equation}
where
\begin{equation}\label{eq:iprc_exact}
\begin{split}
 B_1 &= M_{2} e^{A_1 t_1}\\
 B_2 &= M_{3} e^{A_2 t_2} M_{2} e^{A_1 t_1}\\
 \vdots\\
 B_{K-1} &= M_{K} e^{A_{K-1} t_{K-1}} \cdots M_{3} e^{A_2 t_2} M_{2} e^{A_1 t_1},
\end{split}
\end{equation}
and 
\ymp{\begin{equation}
 T_k = \sum_{i=1}^k t_i, \quad k=1,\ldots,K,
\end{equation}
where $t_i$ denotes the time of flight of the $i^{th}$ portion of the limit cycle.}
\end{corollary}
\yp{\textbf{Proof outline}: The iPRC solution is a result of iterating the adjoint equation solution and jump matrices forward in time. See \S \ref{section:proof-of-corollaryIII3} for the proof of Corollary \ref{corl:eqn}.}

\bl{\begin{remark}
For examples of the matrices $M_{k+1}$ see Eqs.~\eqref{eq:jump_glass} (Glass network), \cite{park2013infinitesimal} (Iris system), and \eqref{eq:shc-m1}, \eqref{eq:shc-m2}, \eqref{eq:shc-m3} (3D piecewise linear central pattern generator).
\end{remark}}

Shaw \textit{et al.}~2012 considered a piecewise linear system satisfying \yp{Assumptions \ref{a:limit_cycle}--\ref{a:direct_deriv}}, where the vector field is piecewise differentiable but discontinuous at subdomain boundaries. Our theory correctly captures all discontinuities of the iPRC at the boundaries \cite{ShawParkChielThomas2012SIADS}. Coombes (2008) considered systems with continuous vector fields not necessarily differentiable at domain boundaries, and analytically computed continuous iPRCs for each system. In the following corollary, we show that a continuous iPRC is a general property of limit cycles over continuous, piecewise smooth vector fields.

\begin{corollary}\label{corl:c0-vector-fields}
  Under the assumptions of Corollary \ref{corl:circuit}, if adjacent vector fields evaluated along the limit cycle, $\bm{F}_{k+1,0}$ and $\bm{F}_{k,t_k}$, are continuous at the boundary $\bdypt_{k+1}$, then the matrix $M_{k+1}$ is the identity matrix.
\end{corollary}
\yp{\textbf{Proof outline}: If the vector field is continuous, then the solution velocity before the crossing ($F^-$) is equal to the velocity after the crossing ($F^+$). Since all other terms in Equation \eqref{eq:saltation} are the same, and the jump matrix is the product of a matrix times its inverse, the jump matrix reduces to the identity. See \S \ref{section:proof-of-corollaryIII4} for the proof of Corollary \ref{corl:c0-vector-fields}.}

\begin{remark}
Our analysis therefore includes the iPRC calculations of \cite{Coombes:2008:SIADS} and \cite{ShawParkChielThomas2012SIADS}  as special cases.
\end{remark}

%  \yp{a planar piecewise constant system where the nonlinearities arise strictly from the boundaries in \S 3.1, a planar piecewise linear oscillator introduced in a motor control context \cite{ShawParkChielThomas2012SIADS}, but generalized here to a non-symmetric geometry in \S 3.2, a piecewise linear genetic regulatory circuit model (Glass network \cite{GlassPasternack1978JMB,GlassPerez1974JChemPhys}) in \S 3.3, a three-dimensional motor control model \cite{ShawLyttleGillCullinsMcManusLuThomasChiel2014JCNS} in \S 3.4, a four-dimensional weakly diffusively coupled version of the piecewise constant system in \S 3.4.1, and a six-dimensional threshold linear network model comprising of two weakly coupled three-dimensional oscillators \cite{MorrisonDegeratuItskovCurto2016arXiv} in \S 3.4.2.}  In \S 4.1 we discuss the relation between our boundary-crossing correction matrix and the classical saltation matrix, in \S 4.2 we discuss the limitations of the method, and  in \S 4.3 we discuss a range of possible further applications.  Following our conclusion  \S 5, the appendices detail the proofs and derivations of the  results.

\section{Results}\label{sec:examples}
\yp{We apply our analysis to several examples in order of increasing complexity. The first is a piecewise constant system with the vector fields arranged such that a limit cycle exists (\S 3.1). The second example (\S 3.2) is a planar system introduced in \cite{ShawParkChielThomas2012SIADS}, motivated by investigations of heteroclinic channels as a dynamical architecture for motor control. The third example (\S 3.3) 2D Glass network, a piecewise linear system obtained as the singular limit of a class of models for feedback inhibition and gene regulatory networks \cite{GlassPasternack1978JMB}. The fourth example (\S 3.4) is a 3D piecewise linear system arising as a simplification of a nominal central pattern generator model for regulation of feeding motor activity in the marine mollusk \textit{Aplysia californica} \cite{LyttleGillShawThomasChiel2017BiolCyb,ShawLyttleGillCullinsMcManusLuThomasChiel2015JCNS} and related to a Lotka-Volterra system with three populations \cite{NowotnyRabinovitch2007PRL}. In the final two examples, we extend our theory to the case of weakly coupled oscillators (\S 4.1). In the first weakly coupled example (\S 3.4.1), we take a pair of weakly coupled piecewise constant models from \S 3.1 (for a total of four dimensions) and analyze the existence, stability, and time-course of synchronous and anti-phase states. Finally, in \S 3.4.2, we consider a pair of weakly coupled piecewise linear models (a total of six dimensions), and show that the weak coupling theory accurately predicts the time-course and synchronization properties of the network.}

\subsection{\yp{Piecewise Constant Model}}\label{ssec:oct}
\yp{To highlight the necessity of our theory, we explore a simple planar example consisting solely of piecewise \emph{constant} vector fields. Through na\"{i}ve application of classic theory, one would expect that because the Jacobian matrix of each region is zero (due to the constant nature of each region), that the iPRC would also be constant. However, this is not the case as shown by numerical simulations (Figure \ref{fig:oct_prc}).}

\revone{Limit cycles -- closed, isolated periodic orbits -- do not occur in linear dynamical systems.  It is particularly striking, therefore, when they arise in a Filippov system, the right hand side of which is piecewise constant. In Figure \ref{fig:oct}, we show the domain of the system with example trajectories that converge  to the limit cycle (purple). \revonealso{On the right panel of the same figure, we define the vector field and corresponding domain}.}

\begin{figure}[h!]
\begin{minipage}[c]{0.5\textwidth}
\includegraphics[width=\linewidth,keepaspectratio=true]{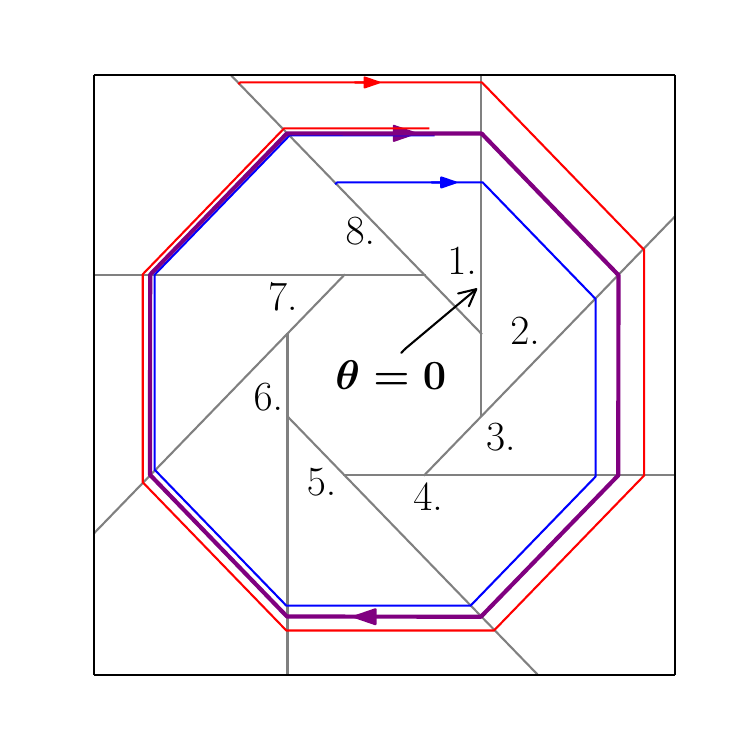}
\end{minipage}%%%%%%%%%%%%%%%%
\begin{minipage}[c]{0.5\textwidth}
\scriptsize
\begin{align*}
 f(\x) = 16\left\{
 \begin{array}{ll}
          \begin{pmatrix}
          1\\0
          \end{pmatrix}, &\x \in \{-y+\sqrt{2} \leq x < 1 \} \\
          \begin{pmatrix}
          1/\sqrt{2}\\-1/\sqrt{2}
          \end{pmatrix}, &\x \in \{1 \leq x < y + \sqrt{2}\} \\
          \begin{pmatrix}
          0\\-1
          \end{pmatrix}, &\x \in\{-1 \leq y < x - \sqrt{2}\}\\
          \begin{pmatrix}
          -1/\sqrt{2}\\-1/\sqrt{2}
          \end{pmatrix}, &\x \in\{-x-\sqrt{2} \leq y < -1\} \\
          \begin{pmatrix}
          -1\\0
          \end{pmatrix}, &\x \in\{-1\leq x < -y-\sqrt{2}\}\\
          \begin{pmatrix}
          -1/\sqrt{2}\\1/\sqrt{2}
          \end{pmatrix}, &\x \in\{ y-\sqrt{2}\leq x < -1\} \\
          \begin{pmatrix}
          0\\1
          \end{pmatrix}), &\x \in\{x+\sqrt{2}\leq y < 1\}\\
          \begin{pmatrix}
          1/\sqrt{2}\\1/\sqrt{2}
          \end{pmatrix}, &\x \in\{1\leq y < -x+\sqrt{2}\} \\
 \end{array}\right.
\end{align*}
\null
\par\xdef\tpd{\the\prevdepth}
\end{minipage}
\caption{\revone{The domain and solutions of the piecewise constant system. On the left panel, domains are labeled with numbers 1--8. We choose zero phase to be the vertical boundary between regions 1 and 2. The limit cycle solution is shown in purple, while example trajectories inside and outside the limit cycle are shown in blue and red, respectively. The right panel specifies the dynamics of each region with the corresponding domain. We define each set such that it implicitly contains all $(x,y)\in\mathbb{R}^2$ satisfying the inequalities. For all remaining  $\mathbf{x}$ (\textit{i.e.}, within the central octagonal region) we use $d\mathbf{x}/dt=\mathbf{x}$.}}
\label{fig:oct}
\end{figure}

Traditional application of the adjoint method suggests, incorrectly, that the solution to the adjoint equation is constant, because the Jacobian matrix in each region is zero. Standard numerical calculation of the iPRC (dots in Figure \ref{fig:oct_prc}) shows that the traditional approach is not sufficient. The numerics suggest that the iPRC is constant except on a set of measure zero, where it jumps discontinuously to other values. 

\yp{As a  check of our analytical results, we numerically evaluate the iPRC at a given phase as follows.} We integrate the limit cycle up to the given phase, then reinitialize the solution at a new position in the direction of a standard basis vector of magnitude between 1e-2 and 1e-4. This process effectively applies an infinitesimal delta function perturbation in phase space. After integrating for a sufficiently long time (typically 10 times the period), we record the timing difference between the unperturbed and perturbed limit cycles. We then divide this value by the magnitude of the perturbation. This magnitude is sufficiently small to give an accurate approximation to the true iPRC for the systems here. 

Using Theorem \ref{theorem}, we compute the size of the discontinuities exactly. The iPRC of both coordinates are shown in Figure \ref{fig:oct_prc}. Note that the Jacobian is identically zero within each domain, so the stability of the limit cycle arises from the contraction of adjacent trajectories (by a factor of $1/\sqrt{2}$) at each boundary crossing. 

\begin{figure}[h!]
\centering
\includegraphics[width=.75\linewidth]{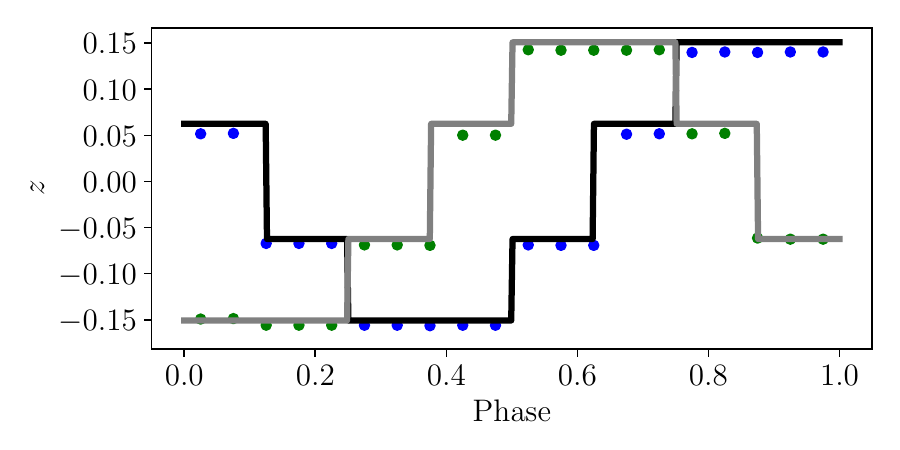}
 \caption{\revone{iPRC of the piecewise constant model. \yp{The numerically computed iPRC is shown in blue (green) for the first (second) coordinate}. The first (second) coordinate of the analytically computed iPRC function is shown in black (gray). The piecewise constant vector field leads to a piecewise constant iPRC. The components take the values $\pm\frac1{16}=\pm 0.0625$ and $\pm 1/(16(\sqrt{2}-1))\approx \pm 0.1509$.}}
 \label{fig:oct_prc}
\end{figure}

\yp{This example highlights the necessity of our theorem to compute the iPRC for oscillators in piecewise smooth systems. However, application of Theorem \ref{theorem} extends well beyond this simple, constructed example. In the examples to follow, we show the broad applicability of our theorem by analyzing several existing models in literature. We begin with a symmetric piecewise linear model from \cite{ShawParkChielThomas2012SIADS}.}

\subsection{\yp{Piecewise Linear Iris System With Non-Uniform Saddle Values}}
 %\todo{see aplysia /papers /yxp30 /j-math-neuro /iris\_modified.py}
In this section, we briefly discuss a planar model for which the iPRC is explicitly computable \cite{ShawParkChielThomas2012SIADS}, and how our theory extends this calculation.

As part of \cite{ShawParkChielThomas2012SIADS}, we analyzed a piecewise linear system (the ``iris system''), where a stable heteroclinic cycle (SHC) gives rise to a one-parameter family of limit cycles similar to those in the sine system. The piecewise linear analogue consists of four regions with velocity fields equivalent, under successive 90-degree rotations, to the velocity field in the unit square. Our theory reproduces the result immediately as shown in \cite{park2013infinitesimal}. Moreover, while the result in \cite{ShawParkChielThomas2012SIADS} depends on symmetry of the system, Theorem \ref{theorem} naturally accounts for the same system with heterogeneous vector fields,
\begin{equation}
\begin{split}
  \frac{d\hat{s}_k}{dt} &= -\lambda_k \hat{s}_k,\\
  \frac{d\hat{u}_k}{dt} &= \hat{u}_k,
\end{split}
\end{equation}
where $\lambda_k$ for $k \in \{1,2,3,4\}$. Despite the loss of symmetry, the calculations in this case follow the same steps as in \cite{park2013infinitesimal}. \yp{Next, we turn to a classic, biologically-motivated model with asymmetric piecewise linear vector field dynamics.}

\subsection{\yp{Two Dimensional Glass Network}}
Glass, Perez, and Pasternack introduced a planar piecewise linear system as a model of feedback inhibition in a genetic regulatory circuit  \cite{GlassPasternack1978JMB,GlassPerez1974JChemPhys}. Figure \ref{fig:glass_pasternack}(a) illustrates several trajectories converging to a stable limit cycle in such a network. The concentration $x_1$ stimulates the production of $x_2$, while $x_2$ inhibits the production of $x_1$. %\todo{see aplysia /papers /yxp30 /j-math-neuro /glass-pasternack-2d.py for figures of adjoint equation solutions} 
One may also consider a macroscopic analogue in a predator-prey system, where $x_1$ is the prey and $x_2$ is the predator. Within each quadrant of the plane, the trajectories converge towards a target point located in the subsequent quadrant (Figure \ref{fig:glass_pasternack}(a)).
\begin{figure}
\includegraphics[width=.8\linewidth]{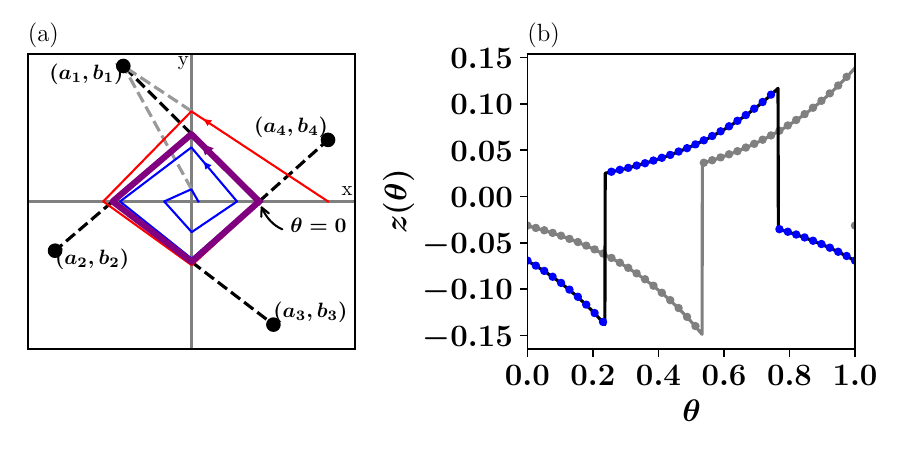}\,\\
\caption{\yp{\textbf{(a):} A model of feedback inhibition as discussed in Example 1 of Glass and Pasternack 1978. The limit cycle attractor (purple) traverses four quadrants, which serve as the four domains of the model.  We call the first quadrant region 1, and because the solutions travel counter-clockwise, the second quadrant is named region 2, the third quadrant region 3, and the fourth quadrant region 4.  Within each quadrant, trajectories are attracted to a target point outside the domain, as shown by the black and gray dashed lines.  For example, in region 1, the limit cycle trajectory (purple) is attracted to the target point $(a_1,b_1)$ until it hits the positive $y$-axis, at which point the limit cycle trajectory changes direction towards the next target point, $(a_2,b_2)$.  Two sample trajectories in the same region, one inside (blue) and one outside (red) the limit cycle, demonstrate that the purple loop is indeed a limit cycle attractor.  The target points are $(a_1,b_1) = (-5,11)$, $(a_2,b_2) = 
(-10,-4)$, $(a_3,b_3) = (6,-10)$, and $(a_4,b_4) = (10,5).$ \textbf{(b):} The numerical (dots) and analytical (lines) infinitesimal phase response curves of the planar Glass network model. The analytical solution to the adjoint equation is given by Eq.~\eqref{eq:iprc_exact}; the numerical iPRC is calculated via direct perturbation. Blue curve: iPRC for perturbations along the horizontal axis. Gray curve: iPRC for perturbations along the vertical axis.}}
\label{fig:glass_pasternack}
\end{figure}
%\todo{for figure, see aplysia /yxp30 /papers /j-math-neuro /glass\_2d\_fig.pdf}
%

The limit cycle attractor, its time of flight through each region, and its points of intersection with each axis is found by using Poincar\'e maps. We refer the reader to the references above for the details. By Theorem \ref{theorem}, the jump matrices are
\begin{equation}\label{eq:jump_glass}
 \begin{split}
  M_1 = \left ( \begin{matrix} 1 & 0 \\ \frac{ a_4-a_1}{b_1}& \frac{b_4}{b_1} \end{matrix} \right ), &\quad M_{2} = \left ( \begin{matrix}  \frac{ a_1}{a_2}& \frac{b_1-b_2}{a_2} \\ 1 & 0 \end{matrix} \right ),\\
  M_{3} = \left ( \begin{matrix} 1 & 0 \\ \frac{ a_2-a_3}{b_3}& \frac{b_2}{b_3} \end{matrix} \right ), &\quad M_{4} = \left ( \begin{matrix}  \frac{ a_3}{a_4}& \frac{b_3-b_4}{a_4} \\ 1 & 0 \end{matrix} \right ).
 \end{split}
\end{equation}

By Corollary \ref{corl:circuit}, we find that the eigenvector associated with the unit eigenvalue is
\begin{equation}
\bm{z}_{1,0} = \frac{\hat{\bm{z}}_{1,0}}{T \left[ b_1-\frac{a_1}{b_1}(a_1-\left[\bdypt_{1} \right]_1) \right ]}.
\end{equation}
We now have enough information to generate the iPRC using Equation \eqref{eq:iprc_exact}, which we show in Figure \ref{fig:glass_pasternack}(b).

% \begin{figure}[h!]
% \includegraphics[width=\linewidth]{glass_2d_prc_fig.pdf}
% \caption{The numerical (dots) and analytical (lines) infinitesimal phase response curves of the planar Glass network model. The analytical solution to the adjoint equation is given by Eq.~\eqref{eq:iPRC_for_Glass_network}; the numerical iPRC is calculated via direct perturbation. Blue curve: iPRC for perturbations along the horizontal axis. Gray curve: iPRC for perturbations along the vertical axis.}
% \label{fig:iPRC_compare_Glass}
% \end{figure}
 % \todo{for figure, see aplysia /papers /yxp30 /j-math-neuro /glass\_2d \_prc\_fig.pdf}

\revonealso{Figure \ref{fig:glass_pasternack}(b) shows the iPRC obtained analytically together with the iPRC obtained by direct numerical perturbations. The iPRC components show the sign and size of the effect that a small displacement away from the limit cycle (LC) trajectory has on the subsequent timing of the trajectory as it returns to the LC.  For example, the descending blue curve (horizontal component) reaches a minimum \yp{at the point of egress from the northeastern sector}, with a discontinuity occurring where the LC trajectory crosses the ray $x=0,y>0$ (\textit{cf.} Figure \ref{fig:glass_pasternack}(b)). The timing of the limit cycle has its greatest sensitivity to small perturbations in the $(1,0)$ direction at this point in the cycle; a perturbation in this direction causes a delay in the trajectory upon return to the LC. Similarly, the greatest sensitivity to perturbations in the horizontal direction occur immediately before crossing the ray $y=0,x<0$; a perturbation in the direction $(0,1)$ at this point causes a significant delay. In contrast, the greatest phase advance in response to a horizontal (resp., vertical) perturbation occurs just before the border crossing at approximately 3/4 of the period (resp., 1 full period).
}

\yp{So far our examples have been scalar or planar, but Theorem \ref{theorem} applies to systems in arbitrary dimensions. In our next example, we consider a novel piecewise linear system in three dimensions.}

\subsection{\yp{Nominal Biting Model of \textit{Aplysia Californica} }}

In this section, we consider a piecewise linear analogue of the three-dimensional system considered in \cite{ShawLyttleGillCullinsMcManusLuThomasChiel2015JCNS}. Written in the order of regions 1, 2, and 3, respectively, we consider the following system:
\begin{equation}\label{eq:piecewise-linear-heteroclinic-full}
\frac{d\bm{r}}{dt} = \left \{ \begin{array}{cl}
  \begin{matrix}1-x -(y+a) \rho \\ y + a \\ (z-a)(1-\rho)\end{matrix},& \quad x \geq y+a, x \geq z-a,\\
  \vspace{5pt}\\
  \begin{matrix}(x-a)(1-\rho) \\ 1-y-(z+a) \rho  \\ z + a\end{matrix},& \quad  y > x-a, y \geq z+a,\\
  \vspace{5pt}\\
  \begin{matrix}x+a\\ (y-a)(1-\rho)\\1 -z - (x+a)\rho \end{matrix},& \quad z > x+a, z > y-a, \\
\end{array}\right .
\end{equation}
where $\bm{r} = (x,y,z)$, $a\ge 0$ is the bifurcation parameter, and $\rho$ is the coupling strength.  The domains of Eq.~\eqref{eq:piecewise-linear-heteroclinic-full} lie in equal thirds of the unit cube, which, when $a=0$, all share an edge along the vector $(1,1,1)$.  For  $a=0$, the domain of region 1 is the convex hull of the vertices $(1,0,0)$, $(1,0,1)$, $(1,1,0)$, and $(1,1,1)$.  Similarly, the domain of region 2 is the convex hull of the vertices $(0,1,0)$, $(1,1,0)$, $(0,1,1)$, and $(1,1,1)$, and the domain of region 3 is defined by the vertices $(0,0,1)$, $(1,0,1)$, $(0,1,1)$, and $(1,1,1)$.  The saddle points of the system lie on a vertex of each domain, namely at $(1,0,0)$, $(0,1,0)$, and $(0,0,1)$, for regions 1, 2, and 3, respectively.

The calculation of the jump matrices is more involved than the previous examples, so we include additional details in Appendix \ref{a:aplysia_details} and move ahead to the calculation of the initial value of the iPRC.

\begin{figure*}
\includegraphics[width=\textwidth]{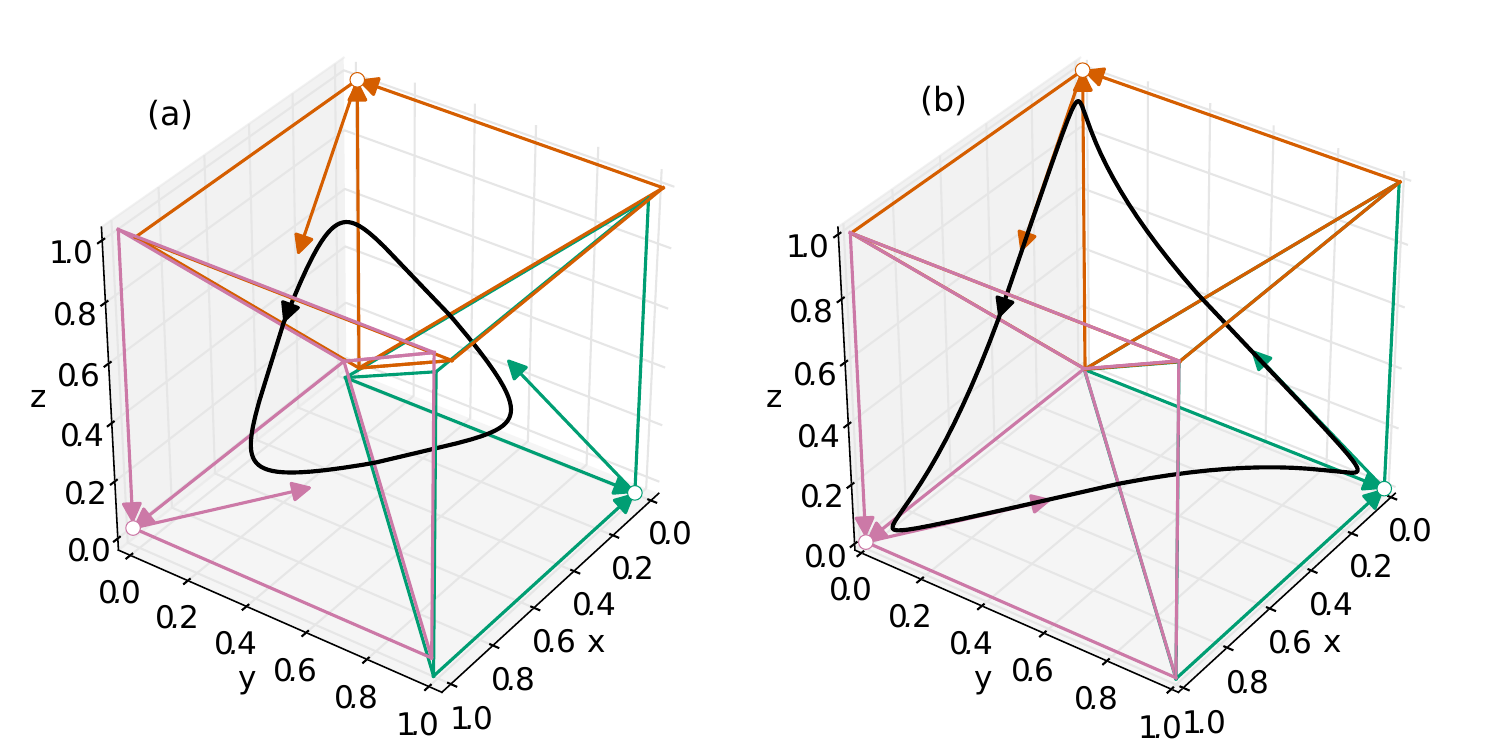}
 \caption{Piecewise linear model of \textit{Aplysia} motor system for two bifurcation parameter values,  \ymp{$\rho=3$, $a=0.02$ (a) and $a=0.001$} (b)  The three domain boundaries are defined according to Eq.~\eqref{eq:piecewise-linear-heteroclinic-full}, with $\| (x,y,z) \|_1 \leq 1$ when $a=0$.  Regions 1, 2, and 3 are colored magenta, green, and orange, respectively (color online).  Each region includes one saddle point, denoted by a circle.  The two arrows pointing into the saddle point 
  indicate the two eigenvectors generating the stable manifold, and the arrow pointing away from the saddle 
  indicates the  eigenvector generating the unstable manifold (Eqs.~\eqref{eq:biting_eigenvectors1}-\eqref{eq:biting_eigenvectors3}).
The black loop in both figures represents the stable limit cycle, with the black arrow denoting the direction of flow.   The tip of the black arrow marks the point $\bdypt_{1}$ at the boundary between regions 3 and 1.  We define the phase at this point to be zero.}
 \label{fig:nominal-biting}
\end{figure*}
 %\todo{for figure, see aplysia /papers /yxp30 /j-math-neuro /nominal\_biting\_fig\_combined.pdf}
%
We continue with $\rho=3$ and \ymp{$a =0.01$}, which is sufficient for a limit cycle solution to exist (Figure \ref{fig:nominal-biting}). By Corollary \ref{corl:circuit}, $B$ is given by
\begin{equation}
B \approx 10 \times \left ( \begin{matrix}
1.25\times 10^3 & 2.71 & -4.23 \times 10^2\\
 2.09 \times 10^4 & 3.18 \times 10  & -2.5 \times 10^3\\
1.52 \times 10^3 & 2.90  &  -3.85 \times 10^2
\end{matrix} \right),
\end{equation}
with a near-unit eigenvalue of approximately $0.998$.  The associated eigenvector, 
%This eigenvalue is not exactly $1$ because of numerical errors, but we assume the errors are sufficiently small for our analysis to hold.  The associated eigenvector, 
$\hat{\bm{z}}_{1,0}$ is,
\begin{equation}
\hat{\bm{z}}_{1,0} \approx \left ( 1.15 \times 10^{-3}, -1, -2.98\times10^{-3} \right ).
\end{equation}

As in the preceding examples, the initial condition of the iPRC, $\bm{z}_{1,0}$, comes from scaling this eigenvector of matrix $B$, $\hat{\bm{z}}_{1,0}$, by Eq.~\eqref{eq:normalization} of Corollary \ref{corl:circuit}:
\begin{equation}
 \bm{z}_{1,0} = \frac{\hat{\bm{z}}_{1,0}}{(\hat{\bm{z}}_{1,0}\cdot \bdypt_{1} )T}.
\end{equation}
The values $\bdypt_{1}$ and $T$ represent the initial condition of the limit cycle and the total period of the limit cycle, respectively, and are found numerically.
\begin{figure}[h!] 
\includegraphics[width=.75\linewidth]{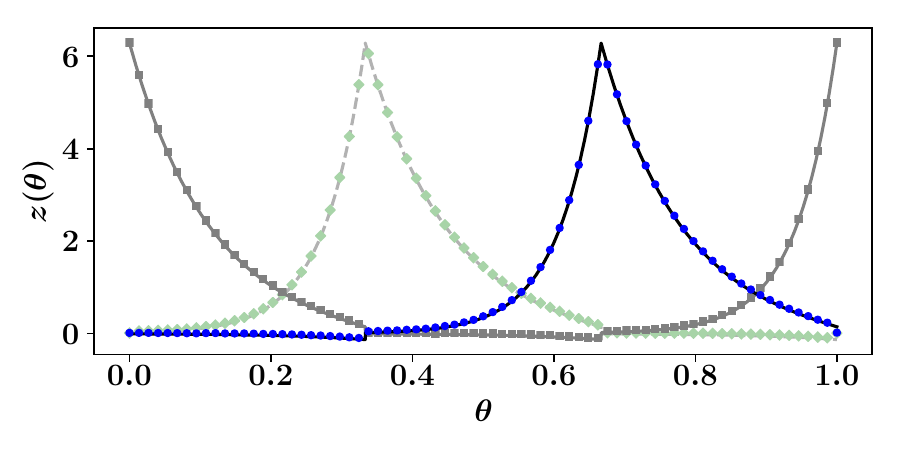}
 \caption{\textit{Aplysia} motor control model iPRC, for  parameters $\rho=3$ and  $a=0.01$.  The blue dots, gray squares, and light green diamonds represent the numerical iPRC, found by the direct method of perturbations, of the first, second, and third components of the iPRC, respectively. The solid black, solid gray, and dashed light gray lines represent the analytical iPRC derived using Theorem \ref{theorem} and Corollary \ref{corl:circuit} of the first, second, and third components of the iPRC, respectively. }
 \label{fig:nominal-biting-prc}
\end{figure}
 %\todo{see aplysia /papers /yxp30 /j-math-neuro /nominal\_biting \_prc\_fig.pdf}

% \begin{figure}[h!]
%  \centering
%  \includegraphics[width=\textwidth]{piecewise_linear_heteroclinic.pdf}
% \end{figure}

\revonealso{For the numerical iPRC, we choose the perturbation magnitude to be $10^{-4}$, with width equal to the time step and follow the same procedure used to generate Figure \ref{fig:glass_pasternack}(b). Figure~\ref{fig:nominal-biting-prc} plots the resulting analytic iPRC together with the iPRC obtained numerically by the direct perturbation method, showing good agreement.}
\revonealso{The greatest phase advance in response to a positive perturbation in the $(1,0,0)$ direction (black), $(0,1,0)$ direction (solid gray), and $(0,0,1)$ direction (light dashed gray lines) occurs at phase 2/3 (entry to the region including the point $(0,0,1)$), phase 0 (entry to the solid region including the point $(1,0,0)$), and phase 1/3 (entry to the region including point $(0,1,0)$), respectively.
}

\yp{Each of the preceding examples is consistent with our theory. We now turn to more practical applications. In the final pair of examples to follow, we demonstrate the utility of Theorem \ref{theorem} and iPRCs in general by analyzing weakly coupled piecewise smooth oscillators in 4- and 6-dimensions, respectively.}

% The results to follow are complementary to \cite{shirasaka2017phase}.

\subsection{\revone{Synchronization of Piecewise Smooth Oscillators}}
\label{ssec:sync}
\revone{In the classic theory for smooth systems, synchronization properties of weakly coupled identical oscillators of the form
\begin{align}
 \frac{d\x}{dt} &= f(\x) + \ve G(\x,\y)\label{eq:x}\\
 \frac{d\y}{dt} &= f(\y) + \ve G(\y,\x)\label{eq:y},
\end{align}
where $|\epsilon|\ll 1$,
may be understood by studying the so-called \emph{interaction function}, $H$, which is the convolution of the coupling function $G$ with the iPRC of the $T$-periodic limit cycle oscillator $U$ satisfying $\dot U = f(U)$. Thus \begin{equation}\label{eq:hfun}
 H(\phi) = \frac{1}{T}\int_0^T z(t)\cdot G[U(t),U(t+\phi)]dt,
\end{equation}
where $U$ is the limit cycle oscillator, and $z$ is the associated (vector) iPRC  \cite{ErmentroutTerman2010book}. Although the theory assumes that $U$ satisfies both Equations \eqref{eq:x} and \eqref{eq:y} for $\ve$ small, the identical oscillators may have different phases, $\theta_1$ and $\theta_2$, respectively. To study synchronization, we consider the phase difference $\psi = \theta_2 - \theta_1$, \textit{i.e.}~we must integrate the scalar differential equation
\begin{equation}\label{eq:phase_diff}
 \frac{d\psi}{dt} = \ve [H(-\psi) - H(\psi)].
\end{equation}
Stable fixed points of this equation predict the existence and type of stable phase-locked solutions under weak coupling.}
\pt{\begin{remark} 
\label{rem:constant_z}
If $z(t)$ were constant, then $H(\phi)$ would be constant, and $d\psi/dt$ would be identically zero.\end{remark} We will refer back to this elementary observation in \S \ref{sssec:piecewise_constant_LC}.}

\revone{As established above, piecewise smooth oscillators may have 
the infinitesimal phase response curves with discontinuities at
domain boundaries.  By obtaining the form of such discontinuities,
our theory allows us to extend weakly coupled oscillator analysis
to weakly coupled piecewise smooth oscillators.  We present two
examples to illustrate this application of the theory:
first, we consider coupled limit cycle oscillators in a recently introduced threshold linear network framework \cite{MorrisonDegeratuItskovCurto2016arXiv}, and show that our analysis correctly captures not only the correct stable phase-locked state, but also the time course of synchronization.  Even more strikingly, we demonstrate synchronization of coupled limit cycle oscillators in a \emph{piecewise constant} dynamical system.  For a piecewise constant system, the Jacobian matrix is identically zero in the interior of each domain. Hence all nonlinear effects -- including the existence of limit cycles and their synchronization properties -- may be said to arise entirely through the boundary-crossing behavior analyzed in this paper.}

For each example below, we show that our calculations correctly predict not only the type of synchronization, but also the time course of synchronization, demonstrating that we correctly capture the nonlinear dynamics of synchronization in these weakly coupled piecewise smooth limit cycle systems. We begin by coupling a pair of one of our first examples, the piecewise constant system (Figure \ref{fig:oct}).

\subsubsection{\revone{\yp{Synchronization of Weakly Coupling of Oscillators with Piecewise-Constant--Velocity}}}
\label{sssec:piecewise_constant_LC}

As seen in Section \ref{ssec:oct}, the limit cycle of the piecewise constant system arises as a result of the boundaries. Similarly, the iPRC is entirely determined by the boundary effects, and is piecewise constant in the domain interiors. Synchronization effects arise therefore due to the nonlinear effects at the boundaries and can be analyzed using Theorem \ref{theorem}.

\revone{We study the synchronization of two such systems under diffusive coupling of the form  $\dot \x = f(\x) + \ve G(\x,\y),\quad 
\dot \y = f(\y) + \ve G(\y,\x),$
%\begin{align*}
% \dot \x = f(\x) + \ve G(\x,\y),\\
% \dot \y = f(\y) + \ve G(\y,\x),
%\end{align*}
\revonealso{where $\x = (x_1,x_2)^T$, $\y=(y_1,y_2)^T$}, and
\begin{equation}
 G(\x,\y) = g\left(\begin{matrix}
             (y_1-x_1)\\
             (y_2-x_2)
            \end{matrix}\right).
\end{equation}
Using Corollary \ref{corl:circuit}, we find the discontinuities at the boundaries, and analytically derive the iPRC (Figure \ref{fig:oct_prc}). The matrix $B$ is
\begin{equation}
 B = \left( \begin{matrix}
      1 & 0\\
      15(1 + \sqrt{2}) & 16
     \end{matrix}\right),
\end{equation}
which has eigenvalues and associated eigenvectors
\begin{equation}
  \begin{array}{ll}
      \lambda_1 = 1 & \lambda_2=16\\
      \vec v_1 = \left( \begin{matrix}
      1-\sqrt{2} \\
      1
     \end{matrix}\right) &   \vec v_2 = \left( \begin{matrix}
      0 \\
      1
     \end{matrix}\right)
     \end{array}.
\end{equation}}

\revone{Following the calculations in \S \ref{ssec:sync}, we see that the convolution of a piecewise constant iPRC with a linear (diffusive) coupling function gives a $C^1$, piecewise quadratic interaction function (Equation \eqref{eq:hfun}, top right panel of Figure \ref{fig:hfuns}). The bottom right panel of Figure \ref{fig:hfuns} shows the right hand side of Equation \eqref{eq:phase_diff} for this example. Figure \ref{fig:oct_phase} compares the phase difference predicted by the theory (Equation \eqref{eq:phase_diff}) to simulations, showing excellent agreement. The first two panels show the state variables near the beginning and end of the simulation, respectively. The limit cycles (the first limit cycle is shown in shades of orange and the second in shades of dashed blue)  begin close to antiphase, then  converge to an in-phase solution. The rightmost panel shows the predicted and simulated phase difference of the oscillators over time. }
\pt{The horizontal dashed line shows the ``prediction" one would obtain if one neglected the effects due to crossing the switching manifolds, namely the absence of synchronization (i.e.~$d\psi/dt=0$, see  Remark \ref{rem:constant_z}).  Hence synchronization arises solely from the effects of the switching boundaries.}

\begin{figure}[h!]
\centering
\includegraphics[width=\linewidth]{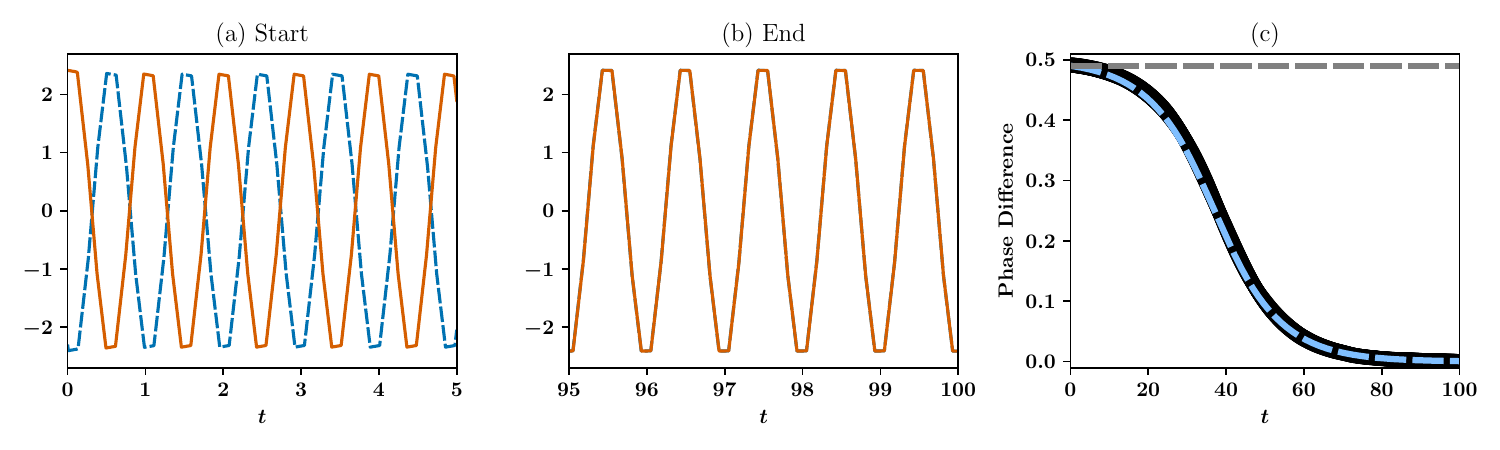}
 \caption{\revone{\yp{Phase difference of weakly coupled piecewise constant oscillators over time. The first state variable of the first oscillator ($x_1$) is orange. The first state variable of the second oscillator ($y_2$) is shown using a dashed blue line. (a) Solutions near the start of the simulation. The dashed blue line is nearly anti-phase to the orange line. (b) Over time, the solutions converge to a synchronous solution. (c) Our predicted phase difference in light blue dashed matches the numerically computed phase difference in black.}}}
 \label{fig:oct_phase}
\end{figure}
% 
% \revone{We conclude this example by establishing analytically the existence of the stable limit cycle suggested by the numerics.
% First, we isolate region 1 into local coordinates,
% \begin{align*}
%  \dot u &= 16,\\ 
%  \dot v &= 0,
% \end{align*}
% which exist within the region bounded by the lines $u=-v$ and $u=0$. These local coordinates and global coordinates are related by the transformation $x = u + 1$ and $y = v + 1/(1+\sqrt{2})$. In local coordinates, an initial condition $(u_0,v_0)$ (where $-v_0 = u_0$ and $u_0 < 0$) on the leftmost edge hits the right edge at location $(0,v_0)$, then resets according to the transformation
% \begin{equation}
%  (0,v_0) \mapsto R\left(0,v_0+\frac{1}{1+\sqrt{2}}\right),
% \end{equation}
% where $R$ is the matrix for the counter-clockwise rotation of a vector by $\pi/4$ radians. The map from the rightmost edge before reset and back to the same edge before reset is
% \begin{equation}
%  f(v_0) = \frac{-2 + \sqrt{2} + v_0}{\sqrt{2}}.
% \end{equation}
% By solving the fixed point equation $f(v_0) = v_0$, we find that $v_0 = \frac{-2\sqrt{2} + 2}{\sqrt{2}-1}$. To determine stability, we take the derivative of $g(v_0) = f(v_0) - v_0$. 
% \begin{equation*}
%  g'(v_0) = \frac{1}{\sqrt{2}} - 1 < 0,
% \end{equation*}
% thus the fixed point is linearly stable, and so the limit cycle is stable as well.}

\begin{figure}
 \includegraphics[width=\textwidth]{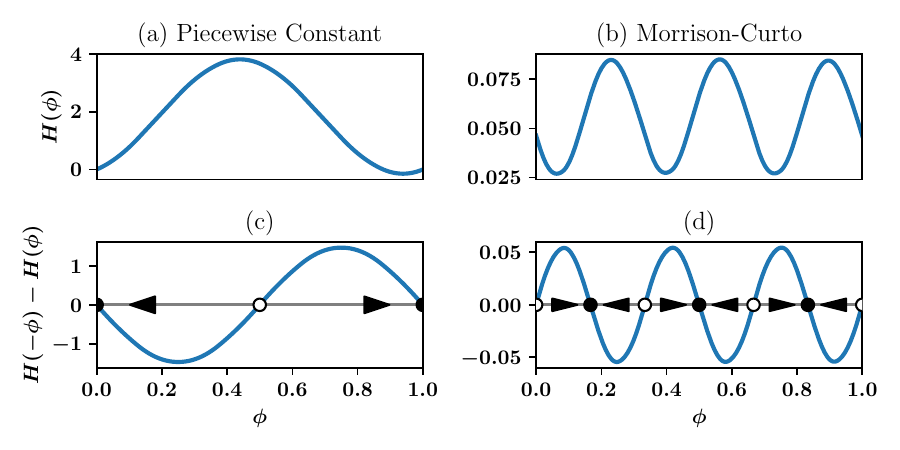}
 \caption{The $H$ functions of the weakly coupled examples. \textbf{(a)}: The $H$ function of the piecewise constant system with reciprocal diffusive coupling. \textbf{(c)}: The right-hand side of the phase dynamics of the piecewise constant system. All non-fixed point solutions tend towards synchrony. \textbf{(b)}: The $H$ function of the Morrison-Curto model. \textbf{(d)}: The right-hand side of the phase dynamics of the Morrison-Curto model. The phase dynamics converge to different phase-locked values depending on the initial conditions. }\label{fig:hfuns}
\end{figure}

\subsubsection{\revone{Synchronization of Weakly Coupled Oscillators with Piecewise-Linear--Velocity}}
\label{sssec:MorrisonCurto}
{In \cite{MorrisonDegeratuItskovCurto2016arXiv} Morrison, Curto and colleagues demonstrate that a simple class of competitive threshold-linear networks can exhibit a rich array of nonlinear dynamical behaviors, including stable  limit cycles, quasi-periodic trajectories, and chaotic dynamics, as well as coexistence of multiple attractor types within networks of modest dimension.  In this section, we study synchronization properties of two weakly coupled limit cycle oscillators within the Morrison-Curto network:
\begin{equation}\label{eq:morrison_curto_full}
\begin{split}
 \frac{dx_i}{dt} &=-x_i+\left [\sum_{j=1}^3 W_{ij} x_j + \alpha(-1-\delta)\sum_{j=1}^3  y_j + \theta\right]_+,\\
 \frac{dy_i}{dt} &=-y_i+\left [\alpha(-1-\delta)\sum_{j=1}^3  \yp{x}_j + \sum_{j=1}^3 W_{ij} x\yp{y}_j + \theta\right]_+,
\end{split}
\end{equation}

with threshold parameter $\theta=1$,
\begin{equation}
W=\left(\begin{matrix}
 0 & -1-\delta & -1+\varepsilon\\
 -1+\varepsilon & 0 & -1-\delta\\
 -1-\delta & -1+\varepsilon & 0
\end{matrix}\right),
\end{equation}
and $0\le\alpha< 1$. The threshold nonlinearity $[\cdot]_+$ is given by $[y]_+ = \max\{y,0\}$. With this choice of weight matrix, and for $\alpha=0$, the system exhibits a limit cycle $\x_0(t)$ \cite{MorrisonDegeratuItskovCurto2016arXiv}. For sufficiently small $\alpha>0$ the limit cycle persists, with negligible changes in its shape and timing properties as $\alpha$ increases.}

\begin{figure}[h!]
\centering
\includegraphics[height=.9\linewidth]{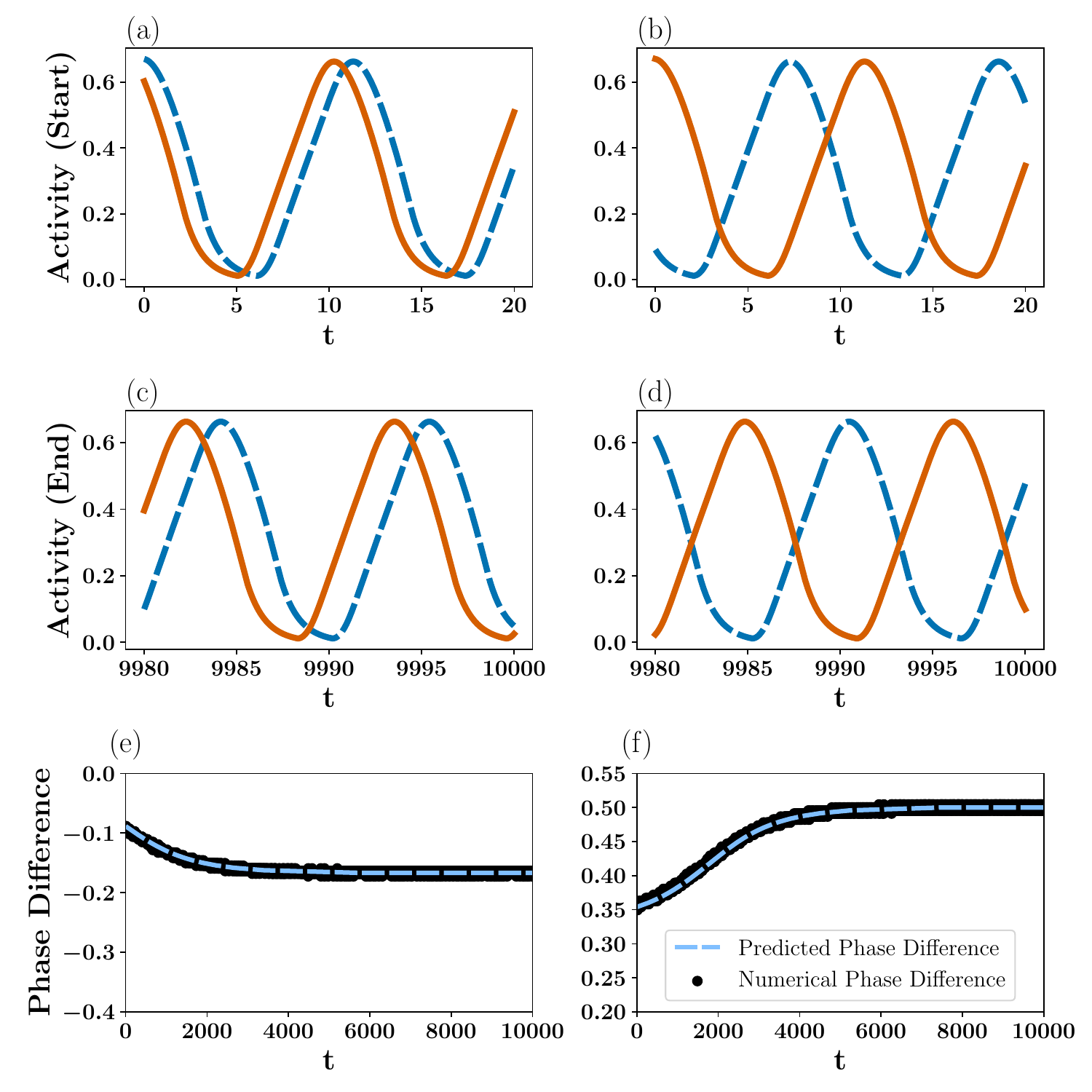}
 \caption{\revone{\yp{Evolution of phase difference for  weakly coupled Morrison-Curto competitive threshold-linear network oscillators. The first state variable of each oscillator are shown in orange and blue dashed (corresponding to $x_1$ and $y_1$, respectively).  \textbf{Left column} ((a),(c),(e)):  near-synchronous initial conditions (a) separate slightly to a state with greater phase lag (c) with a phase difference of $-1/6$ (e). \textbf{Right column} ((b), (d), (f)): near-antisynchronous initial conditions (b) converge to an antiphase solution (d), with a phase difference of $1/2$. \textbf{Bottom row} ((e), (f)): time course of phase difference predicted from reduced phase description, Equation \eqref{eq:phase_diff} (light blue dashed line), matches that of full 6D system (black solid line).  Parameters $\alpha=0.01, \delta=0.5, \varepsilon=0.25$. See also \S \ref{sssec:phase_estimation}.}}}
 \label{fig:threshold}
\end{figure}

{When $\alpha\ll1$, the coupling term $\alpha(-1-\delta)\sum_{j=1}^3 y_j$ has little effect on the boundary crossing points of the limit cycle, but it does contribute to the dynamics of $x_i$, provided the term $\sum_{j=1}^3 W_{ij} x_j$ is above threshold. Thus, to first order in $\alpha$, we may approximate the dynamics as
\begin{equation}
 \frac{dx_i}{dt} =-x_i+\left[\sum_{j=1}^3 W_{ij} x_j+\theta\right]_+  + \alpha(-1-\delta) \Theta\left(\sum_{j=1}^3 W_{ij} x_j+\theta\right)\sum_{j=1}^3 y_j,
\end{equation}
where $\Theta$ is the Heaviside function. In this form, we have a system of weakly coupled oscillators,  $d\x/dt = F(\x) + \alpha G(\x,\y),d\y/dt = F(\y) + \alpha G(\y,\x)$, 
%\begin{align*}
% \frac{d\x}{dt} &= F(\x) + \alpha G(\x,\y),\\
% \frac{d\y}{dt} &= F(\y) + \alpha G(\y,\x),
%\end{align*}
where
\begin{equation}
 F(\x) = \left(\begin{matrix}
          -x_1 + \left[\sum_{j=1}^3 W_{1j} x_j+\theta\right]_+\\
          -x_2 + \left[\sum_{j=1}^3 W_{2j} x_j+\theta\right]_+\\
          -x_3 + \left[\sum_{j=1}^3 W_{3j} x_j+\theta\right]_+
         \end{matrix}\right),
\end{equation}
and
\begin{equation}
 G(\x,\y) = \left(\begin{matrix}
          (-1-\delta) \Theta\left(\sum_{j=1}^3 W_{1j} x_j+\theta\right)\sum_{j=1}^3 y_j\\
          (-1-\delta) \Theta\left(\sum_{j=1}^3 W_{2j} x_j+\theta\right)\sum_{j=1}^3 y_j\\
          (-1-\delta) \Theta\left(\sum_{j=1}^3 W_{3j} x_j+\theta\right)\sum_{j=1}^3 y_j
         \end{matrix}\right).
\end{equation}
In this case, although the vector field is not continuously differentiable at the domain boundaries, it is nevertheless continuous, so by Corollary \ref{corl:c0-vector-fields}, we find that the jump matrices are equal to the identity.  Thus we can establish that for this system the iPRC may be obtained by integrating the adjoint equation in the usual fashion (\textit{cf}.~\S \ref{sssec:piecewise_constant_LC} for an example with nontrivial jump matrices). With the iPRC established, we may numerically integrate Equations \eqref{eq:hfun} and \eqref{eq:phase_diff} to predict the synchronization dynamics of the coupled oscillators. The left column panels of Figure \ref{fig:hfuns} shows the numerically computed $H$ function and right-hand side of the phase difference equation. Figure \ref{fig:threshold} compares the time course of the phase difference $\psi$ predicted by integrating the 1D phase equation \eqref{eq:phase_diff} and the phase difference time course obtained by integrating the full 6D equations \yp{(Equation \eqref{eq:morrison_curto_full})}. The theory and simulations show excellent agreement.}

\subsubsection{\revone{Numerics and Phase Estimation}}
\label{sssec:phase_estimation}
\revonealso{To estimate the phase of an oscillator in the piecewise constant system, we use the geometrical phase angle of the state variables to approximate the phase of the full model. This approximation is reasonable because the system has a high degree of symmetry, and numerical tests reveal little error between this method and the more involved but general method below.

\revone{In the Morrison and Curto model (\S \ref{sssec:MorrisonCurto}), the system lacks the type of symmetry present in the piecewise constant system. Thus, we resort to a brute force lookup table method used in \cite{park2016weakly} to estimate the phase. }}

\revone{For all simulations of the Morrison-Curto model, we use the parameter values $\delta=.5$, $\ve=.25$, $\theta=1$, and $\alpha=.01$. To reduce the number of redundant plot points, we plot every 500 points.}

\revone{All figure code, including data files and files for data generation is available on GitHub at \texttt{https://github.com/youngmp/pwl\_iprc}.}

\section{Discussion}
\label{sec:discuss}

\subsection{\revonealso{Relation to the Saltation Matrix}}
\revonealso{In this paper we have derived the form of the discontinuity in the infinitesimal phase response curve at domain boundaries for a generic limit cycle arising in a piecewise smooth dynamical system \pt{in arbitrary finite dimensions}.  The solution of this problem is closely related to the solution of the variational problem for piecewise smooth systems, \pt{as we now discuss}. %as we now show.

First, to recapitulate our result, consider a trajectory $\gamma(t)$ transversely crossing domain boundary $\Sigma$ at point $\mathbf{p}$ at time $t=0$, exiting the old domain with velocity $\bm{F}^-$ and entering the new domain with  velocity $\bm{F}^+$.  For brevity, write $\mathbf{z}^-$ for the iPRC vector $\lim_{t\to 0^-}\mathbf{z}(t)$ just before the crossing, and $\mathbf{z}^+$ for the iPRC vector $\lim_{t\to 0^+}\mathbf{z}(t)$ immediately after the crossing.  In the interior of domain $k$, the iPRC vector evolves according to $\dot{\mathbf{z}}=-(D\mathbf{F}_k(\gamma(t)))^\intercal\mathbf{z}.$  The boundary crossing induces a linear jump condition $\yp{C}\mathbf{z}^+=\yp{D}\mathbf{z}^-$.
If $\bm{w}_1,\ldots,\bm{w}_{n-1}$ is any orthonormal basis for the tangent space of $\Sigma$ at $\bm{p}$, the matrices \yp{$C$} and \yp{$D$} are given 
\begin{equation}\label{eq:saltation}
\yp{C}=\left(\begin{matrix}\bm{F}^+&|& \bm{w}_1 &|& \cdots &|& \bm{w}_{n-1}\end{matrix}\right)^\intercal,\quad
\yp{D}=\left(\begin{matrix}\bm{F}^-&|& \bm{w}_1 &|& \cdots &|& \bm{w}_{n-1}\end{matrix}\right)^\intercal
,\end{equation}
\textit{cf.}~equation \eqref{eq:matrix-a-b}.

A linear jump condition arises as well in the solution of the variational problem for piecewise smooth systems.  The solution was obtained by  Aizerman and Grantmakher \cite{AizermanGantmakher1958JApplMathMech} and is discussed in the monographs   \cite{BernardoBuddChampneysKowalczyk2008PiecewiseSmoothDynSysBook,LeineNijmeijer2004book}.  Within the interior of the $k$th domain, the evolution of a small perturbation $\bm{u}(t)$ to a trajectory $\bm{y}(t)\approx\gamma(t)+\bm{u}(t)$ evolves, to linear order, as $\dot{\bm{u}}=D\bm{F}_k(\gamma(t))\bm{u}$. Writing $\bm{u}^-=\lim_{t\to 0^-}\bm{u}(t)$ and $\bm{u}^+=\lim_{t\to 0^+}\bm{u}(t)$, the jump in $\bm{u}$, upon $\gamma$ crossing the boundary $\Sigma$ at $\bm{p}$, is given by the \emph{Saltation matrix} $S$.  That is, $\bm{u}^+=S\bm{u}^-$, where the matrix $S=I+\Delta\bm{F}\hat{\bm{n}}^\intercal/\left(\hat{\bm{n}}\cdot\bm{F}^-  \right)$ is defined in terms of the difference in the vector fields, $\Delta\bm{F}\equiv \bm{F}^+-\bm{F}^-$, and the vector $\hat{\bm{n}}$ normal to the surface $\Sigma$ at $\bm{p}$.

Because the basis vectors satisfy $\bm{w}_i\perp\n$ for each $i$, the matrices $A$ and $B$ satisfy the equations
\begin{equation}\label{eq:ABdiscussion}
(S-I)B^\intercal=\left(\begin{matrix}\Delta\bm{F}&|&\bm{0}&|&\cdots&|&\bm{0}\end{matrix}\right)=-(S^{-1}-I)A^\intercal.
\end{equation}
Comparing \eqref{eq:saltation}-\eqref{eq:ABdiscussion}, 
notice that $\yp{C}^\intercal-\yp{D}^\intercal=(S-I)\yp{D}^\intercal$. It follows immediately that the jump matrix can be written in terms of the saltation matrix:
\begin{equation}
M=\yp{C}^{-1}\yp{D}=\left(S^{-1}\right)^\intercal.
\end{equation}
This relation is quite natural:  The saltation matrix accounts for changes in the fate of trajectories with perturbed initial conditions with respect to evolution \emph{forward} in time; it corrects the linearized forward (variational) equation.  The jump matrix obtained in this paper accounts for changes when tracking the prior history of trajectories that would have converged to common points on a limit cycle; it corrects the linearized backward equation for the  effects of boundary crossing.}  

\pt{Alternatively, consider the following elementary derivation.  
\begin{remark}\label{rem:alternative_saltation}
Assume the isochrons are given by a piecewise smooth function $\theta(\mathbf{u})$.  The phase difference between any two points  remains constant in time as the trajectories through those two points evolve.  Thus for any two points $\mathbf{x}$ and $\mathbf{x}+\mathbf{u}$,
\begin{equation}
\frac{d}{dt}\left(\theta(\mathbf{x}+\mathbf{u})-\theta(\mathbf{x})\right)= 0.
\end{equation}
Introducing the gradient $\mathbf{z}(t)\equiv\nabla\theta(\gamma(t))$ and considering an arbitrary small displacement $\mathbf{u}$, we see that $$\mathbf{z}\cdot\mathbf{u}=\textrm{const}$$ along trajectories, both within the interior of a region, and across the switching boundaries.  When $\mathbf{u}$ jumps, in general so does $\mathbf{z}$.   We seek a matrix $M$ satisfying  $\mathbf{z}_+=M\mathbf{z}_-$, where $\mathbf{z}_-$ is the iPRC just before the switching boundary and $\mathbf{z}_+$ is the iPRC just after.  Preserving  $\mathbf{z}\cdot\mathbf{u}$ requires
\begin{equation}
\mathbf{z}_-^\intercal\mathbf{u}_-=\mathbf{z}_+^\intercal\mathbf{u}_+=\mathbf{z}_-^\intercal M^\intercal S \mathbf{u}_-
\end{equation}
for arbitrary small perturbations $ \mathbf{u}$.  Therefore $M^\intercal S= I$, or $M=\left(S^\intercal\right)^{-1}$, as shown above.\end{remark}}

%, which is what Youngmin and I show in our paper (by constructing a basis on the switching surface) and Shirosaka et al also show (although by a rather complicated argument).  Also Coombes uses a similar argument for a specific model in a paper with Thul and Wedgwood. }

\subsection{\ympe{Related Literature}}
\label{ssec:related_literature}
%\ympe{The discussion of saltation matrices is particularly relevant as there exist several works that use this matrix to derive relevant results for piecewise linear and piecewise smooth systems. We discuss these works briefly and their relation to our results.}

\pt{Saltation matrices, or closely related constructions, appear in the analysis of phase response curves in a handful of papers that we now review.}  

\pt{We first derived the general form of the jump condition for the infinitesimal phase response curve of a continuous limit cycle trajectory in arbitrary dimensions in \cite{park2013infinitesimal} without reference to saltation matrix methods. Subsequently, }\ympe{\cite{shirasaka2017phase} used the saltation matrix to account for jumps in the iPRC for general hybrid systems, allowing for discontinuous solutions. \yp{In fact, our result is similar and complementary to \cite{shirasaka2017phase}. In \cite{shirasaka2017phase} and the present study, we require differentiability of the isochron excluding boundaries, and directional differentiability of the isochron within boundaries. This property is shown to be true in \cite{shirasaka2017phase}, justifying Assumptions \ref{a:limit_cycle}--\ref{a:direct_deriv}. The studies diverge at this point: \cite{shirasaka2017phase} reformulates the classic phase reduction approach of Kuramoto \cite{KuramotoBook} as a Filippov system, allowing the application of classic nonsmooth theory. They then analyze the effects of weak, periodic forcing on oscillator entrainment. Our derivation of the jump matrix $M$ and its relation to the saltation matrix \pt{employs} a first principles approach \pt{independent of the} classic nonsmooth theory. We then consider examples of weakly coupled oscillators (as opposed to periodically forced oscillators), and we provide a \pt{rudimentary} derivation of the interaction function \eqref{eq:phase_diff} for weakly coupled oscillators with Filippov dynamics.}}

Other results that derive iPRCs or use the saltation matrix to determine stability of coupled oscillators cover special cases (the most common cases are a combination of planar systems, piecewise linear systems, and relaxation oscillators).  \pt{In cases in which the vector field is continuous across the switching boundary (but not differentiable across it) the saltation and jump matrices are trivial, and can be neglected.  This situation holds in the original piecewise linear version of the Fitzhugh-Nagumo model studied by McKean \cite{McKean1970AdvMath} and later incorporated into network models \cite{Coombes:2008:SIADS}.}

In \cite{CoombesThulWedgwood2012PhysD}, the authors analyze a network of piecewise linear integrate-and-fire models with adaptation. They calculate the discontinuities of the iPRC \yp{using a method equivalent to the saltation matrix approach}. The phase-locking analysis of this paper does not require the coupling strength to be small and \yp{and the solutions are discontinuous}, but requires all-to-all linear coupling between the planar oscillators. In contrast, our result \yp{requires continuous solutions} and allows for non-linear coupling of $n$-dimensional oscillators.

\pt{In some instances,} properties of the limit cycle solution allow a direct derivation of the iPRC using the normalization condition \eqref{eq:normalization} \cite{coombes2001phase,izhikevich2000phase}. In these papers, the systems produce discontinuous solutions in the relaxation limit. The size of the discontinuity, along with the normalization condition, determines the discontinuous iPRC.

The saltation matrix is more commonly used to derive equations distinct from the iPRC, but with similar goals in mind. In \cite{coombes2016synchrony}, the authors derive the master stability function for piecewise linear planar oscillators using the saltation matrix as part of their derivation. They use the saltation matrix to determine the \yp{Floquet} exponents, which determine stability of oscillators in a synchronous state. Again, coupling between oscillators may be strong, but the coupling is linear and the oscillators are required to be identical and the stability analysis only holds for determining stability of the synchronous state. The authors remark, however, that deriving a master stability function for nonlinear coupling and higher dimensional models is possible in future work.

\subsection{\revonealso{Limitations of the Method}}

Our results apply broadly, because many systems have the structure we discuss. Nevertheless some caveats are in order. 

The iPRC does not always  capture the response to a stimulus. In situations in which the linear approximation to the asymptotic phase function breaks down, for instance when the stimulus drives the oscillator's trajectory close to the stable manifold of a saddle point on the boundary of the basin of attraction, mechanisms such as shear-induced chaos can lead to complicated responses to periodic forcing that cannot be predicted \textit{via} iPRC analysis.  This scenario can arise near the homoclinic bifurcation in the Morris-Lecar model, for example \cite{LinWedgwoodCoombesYoung2012JMathBiol}.  \yp{In the limit that the periodic orbit is pushed towards a heteroclinic bifurcation the iPRC can diverge (\textit{cf}.~Figures 6-7), an effect discussed in \cite{ShawParkChielThomas2012SIADS}.} Nevertheless, in many systems the iPRC  plays an important role in understanding oscillator entrainment and synchronization. 

In the examples we consider, the decomposition of the vector field into piecewise linear domains is specified in the statement of the original system. The approximation of limit cycles in a smooth system, with limit cycles in a piecewise linear system, has been investigated in a general setting  \cite{StoraceDeFeo2004CircuitsSystems}.  However, there is no \textit{a priori} heuristic for how to approximate an arbitrary nonlinear system with a piecewise linear approximation.   

\revonealso{Our method requires the phase function to be defined within some open neighborhood of the limit cycle (Assumption 4).  For some nonsmooth systems, it is possible that the phase function may not be well defined throughout the entire basin of attraction of the limit cycle. For example, Simpson and Jeffrey discuss piecewise smooth systems with a ``two-fold singularity" as a mechanism for finite time desynchronization of a limit cycle oscillator \cite{SimpsonJeffrey2016PRSA}.  In their example, the basin of attraction includes regions admitting sliding solutions that remain on the boundary surface for finite times, leading to nonuniqueness of solutions, and hence nonuniquness of the phase function in a portion of the stable manifold of the limit cycle, (although uniqueness is guaranteed within a neighborhood of the limit cycle by the transverse flow condition).  Our construction does not apply to such regions of the basin of attraction. }

\subsection{\revonealso{Further Applications}}
\label{ssec:further_applications}
The extension of results from classical dynamical systems to nonsmooth systems is an active area of research with applications in a wide variety of contexts.  
Carmona \textit{et al.}~studied a canonical form for limit cycles in planar PWL dynamical systems with two regions \cite{carmona_etal_2013}, using Melnikov methods to study existence and bifurcations of limit cycles. Ponce \textit{et al.}~studied bifurcations leading to limit cycles in PWL planar systems in \cite{ponce_etal_2013}. Existence of limit cycles has been shown for planar PWL systems with two regions in  \cite{huan_etal_2012} and \cite{LlibrePonce2012DCDISSBAA}, for planar PWL systems with an arbitrary but finite number of separate regions in \cite{gaiko_van_horssen_2009}, and for a PWL system in $\mathbb{R}^4$ with three regions in \cite{cheng_2013}. Stability of piecewise linear limit cycles in $\mathbb{R}^n$ with $m+1$ regions is analyzed in \cite{lin_etal_2003} using Poincar\'e map techniques.

The review \cite{lin_etal_2009} discusses necessary and sufficient conditions for asymptotic stability of piecewise linear systems in $\mathbb{R}^n$;  \cite{ma_et_al_2013} adapts Lyapunov functions   for piecewise linear systems.

Limit cycles in piecewise linear systems occur not only in biology but also in control engineering  \cite{pettit_1996}.  Piecewise linear systems arise naturally in anti-lock braking systems \cite{morse_1997}, which are themselves engineered to produce limit cycle oscillations \cite{pettit_wellstead_1995}. Piecewise smooth relay feedback systems, first used in heating \cite{hawkins_1887} and more recently in (PID) control \cite{johansson_2002}, can give rise to limit cycles. The exact conditions for limit cycle existence in relay feedback systems is given in \cite{tsypkin_1984}. Many piecewise linear biological models exist as well. A piecewise linear version of the Fitzhugh-Nagumo model (also called the McKean model) and a piecewise linear version of the Morris-Lecar model are studied in \cite{Coombes:2008:SIADS}. The authors in \cite{bizzarri_etal_2007} convert the Hindmarsch-Rose model into a piecewise linear version and analyze its stability. Gene regulatory networks are a classic example of piecewise linear models exhibiting limit cycle oscillations \cite{GlassPasternack1978JMB}, and a subject of ongoing research.
For instance, \cite{edwards_gill_2003} analyzes the stability of synchronous periodic solutions, assuming weak symmetric coupling of two Glass networks. Rigorous investigations of Glass networks have considered them within the framework of differential inclusions \cite{acary_etal_2014,machina_ponosov_2011}. To facilitate construction of networks with customized dynamics \cite{zinovik_etal_2010}\yp{,} systematically classified cyclic attractors on Glass networks with up to six switching units.
\pt{Walsh and colleagues studied periodic orbits in a discontinuous vector field as a model of cycling phenomena in glacial dynamics \cite{walsh2016periodic}.}
In summary, there is a rich collection of contexts in which piecewise linear and piecewise smooth systems arise.

\section{Conclusion}
The infinitesimal phase response curve provides a linear approximation to the geometry of the asymptotic phase function in the  vicinity of a stable limit cycle. The classical method for obtaining iPRCs from the adjoint \cite{ErmentroutKopell:1991:JMathBio} breaks down with nonsmooth dynamics because the Jacobian may not be well defined at the domain boundaries. In this paper we have introduced a general theory for the iPRCs for limit cycles arising in piecewise smooth systems, provided the limit cycle intersects the domain boundaries transversely, the boundaries are smooth at the points of intersection, \yp{and that classic properties of the phase function exist at least within a boundary of the limit cycle (Assumptions \ref{a:limit_cycle}--\ref{a:direct_deriv})}. In the case of piecewise smooth systems which are also continuous across the domain boundaries, we obtain continuous iPRCs.  Discontinuities in the iPRCs may arise when the vector field is discontinuous across domain boundaries, and our analysis provides the explicit form of the discontinuity, \revonealso{in the form of a linear matching condition related to the classical saltation matrix construction.}  Our  results are consistent with, and extend, existing findings \yp{of iPRC calculations}, such as \cite{Coombes:2008:SIADS}.  Because piecewise smooth and piecewise linear systems arise in a wide variety of fields, from biology to engineering, our analysis has the potential for broad application.

\appendix

\section{Proofs of the Main Results}\label{section:proofs-of-main-results}
\subsection{Proof of Theorem \ref{theorem}}\label{section:proof-of-theorem}
From assumptions \ref{a:limit_cycle}--\ref{a:phase},
$\revone{\bm{\gamma}}$ is a piecewise smooth limit cycle that
admits a phase function $\theta(\revone{\bm{x}})$ throughout the basin of attraction (B.A.), such that $d\theta/dt=1/T$ throughout the B.A.  Also by assumption, $\theta$ is differentiable ($C^1$) in the interior of each region and continuous ($C^0$) at the boundaries between successive regions (assumptions \ref{a:isochron} and \ref{a:phase_deriv}).  

For a piecewise smooth dynamical system satisfying hypotheses H1--H4 and assumptions \ref{a:limit_cycle}--\ref{a:sigma}, the crossing point $\bdypt_{k+1}$ lies in an $n-1$ dimensional surface $\Sigma_{k+1}$ that is $C^1$ within a ball $B(\bdypt_{k+1},c)$ of radius $c$, and admits a unique unit length normal vector $\n$ oriented in the direction of the flow.  By Gram--Schmidt, we may construct $n-1$ orthonormal vectors  $\hat{\bm{w}}_i$ spanning the hyperplane tangent to $\Sigma_{k+1}$ at $\bdypt_{k+1}$ orthogonal to $\n$ for all $i=1,\ldots,n-1$.

   \newcommand{\tvi}{\hat{\bm{w}}_{i}}

Introduce local coordinates $u$ on $B(\bdypt_{k+1})\cap \Sigma_{k+1}$ such that $u=0$ corresponds to the point $\bdypt_{k+1}$.
Let $\theta_k$ denote the phase function $\theta$ within the $k$th domain.    
Although the gradient $\nabla\theta$ is not defined at points on the boundary $\Sigma_{k+1}$, we have well defined directional derivatives
\begin{align}
D_{\mathbf{w}}(\theta_k)(\mathbf{x}_k) &= \lim_{h\to 0}\frac{1}{h}\left(\theta(\mathbf{x}_k+h\mathbf{w})-\theta(\mathbf{x}_k)\right)\\
D_{\mathbf{w}}(\theta_k)(\mathbf{x}_{k+1}) &= \lim_{h\to 0}\frac{1}{h}\left(\theta(\mathbf{x}_{k+1}+h\mathbf{w})-\theta(\mathbf{x}_{k+1})\right)
\end{align}
for points $\mathbf{x}_k$ and $\mathbf{x}_{k+1}$ in the interior of region $k$ and region $k+1$, respectively.  Fixing a basis vector $\tvi$ in the plane tangent to $\Sigma_{k+1}$ at $\bdypt_{k+1}$ and taking the limits as $\mathbf{x}_k\to \bdypt_{k+1}$ and $\mathbf{x}_{k+1}\to\bdypt_{k+1}$, we have (Assumption \ref{a:direct_deriv})
 \begin{equation}
    \hat{\bm{w}}_i \cdot \nabla\theta_{k+1}(u) =\hat{\bm{w}}_i \cdot \nabla\theta_k(u) , \quad \forall i=1,\ldots,n-1
  \end{equation}
  (for simplicity, we use $\hat{\bm{w}}_i$ in place of $\tvi^{k+1}$). For $u=0$, the gradient of the phase function is evaluated \yp{on} the limit cycle.  Referring back to the notation of Eq.~\eqref{eq:iprc-limits} we have
  \begin{equation}\label{eq:tangent-direction-equivalence}
\hat{\bm{w}}_i\cdot \bm{z}_{k+1,0} = \hat{\bm{w}}_i\cdot \bm{z}_{k,t_k}, \quad \forall i=1,\ldots,n-1.
\end{equation}
  Eq.~\eqref{eq:tangent-direction-equivalence} provides $n-1$ independent linear equations for the  $n$ unknown values $\bm{z}_j$, $j=1,\ldots n$. To obtain an $n$th independent linear equation, let $m \in \{k,k+1\}$ and let $\bm{x}_m(t)$ be a trajectory in the basin of attraction of vector field $\bm{F}_m$.  By the chain rule,
\begin{equation}
\begin{split}
\frac{d\theta_m}{dt} &=\left [ \nabla\theta_m(\bm{x}_m(t)) \right ] \cdot \frac{d}{dt}\bm{x}_m(t)\\
&= \bm{F}_{m}(x_{m}(t))\cdot\nabla\theta_m(\bm{x}_{m}(t)).
\end{split}
\end{equation}
For a trajectory $\bm{x}_m(t)$  on the limit cycle, i.e., when $\bm{x}_m(t) = \bm{\gamma}_m(t)$, we have by definition of the iPRC (\cite{ErmentroutTerman2010book}),
\begin{equation}
\frac{d\theta_m}{dt} = \bm{F}_{m}(\revone{\bm{\gamma}}_{m}(t))\cdot \bm{z}_{m}(t).
\end{equation}
Recalling that $d\theta/dt=1/T$, taking the one-sided limits,
\begin{equation}\label{eq:phase-limits}
\begin{split}
\lim_{t \rightarrow 0^+} \bm{F}_{k+1}(\revone{\bm{\gamma}}_{k+1}(t))\cdot \bm{z}_{k+1}(t) &= \bm{F}_{k+1,0}\cdot \bm{z}_{k+1,0},\\
  \lim_{t \rightarrow t_k^-} \bm{F}_{k}(\revone{\bm{\gamma}}_{k}(t))\cdot \bm{z}_k(t) &= \bm{F}_{k,t_k}\cdot \bm{z}_{k,t_k},
  \end{split}
\end{equation}
therefore
%yields the final equation relating $\bm{z}_k$ and $\bm{z}_{k+1}$ at the boundary,
\begin{equation}\label{eq:thetaflow}
\bm{F}_{k+1,0}\cdot \bm{z}_{k+1,0} = \frac{1}{T} =  \bm{F}_{k,t_k}\cdot \bm{z}_{k,t_k}.
\end{equation}
Eq.~\eqref{eq:thetaflow} asserts that the phase function advances at the same rate as a function of time everywhere, and in particular on both sides of the boundary point $\bdypt_{k+1}$.  
Combining Eq.~\eqref{eq:thetaflow} with Eqs.~\eqref{eq:tangent-direction-equivalence}
provides $n$ independent linear equations:
\begin{equation}\label{eq:n-iprc-equations}
\begin{split}
\bm{F}_{k+1,0}\cdot \bm{z}_{k+1,0} &=  \bm{F}_{k,t_k}\cdot \bm{z}_{k,t_k},\\
\hat{\bm{w}}_1 \cdot \bm{z}_{k+1,0} &= \hat{\bm{w}}_1 \cdot \bm{z}_{k,t_k},\\
\hat{\bm{w}}_2 \cdot \bm{z}_{k+1,0} &= \hat{\bm{w}}_2 \cdot \bm{z}_{k,t_k},\\
\vdots\\
\hat{\bm{w}}_{n-1} \cdot \bm{z}_{k+1,0} &= \hat{\bm{w}}_{n-1} \cdot \bm{z}_{k,t_k},\\
\end{split}
\end{equation}
which are equivalent to the equality expressed in Theorem \ref{theorem} in terms of the $n \times n$ matrices $\yp{C}_{k+1}$ and $\yp{D}_{k}$ as defined in the theorem.

In order to solve for the initial value of the iPRC of the $k+1$ portion of the limit cycle, $\bm{z}_{k+1,0}$, we must invert the matrix $\yp{C}_{k+1}$.  The invertibility of $\yp{C}_{k+1}$ is guaranteed because the vector field value $\bm{F}_{k+1,0}$ and the $n-1$ tangent vectors are linearly dependent if and only if $\bm{F}_{k+1,0} \in \rm{span}(\hat{\bm{w}}_1,\ldots,\hat{\bm{w}}_{n-1})$.  However, the vector $\bm{F}_{k+1,0}$ can not be written as a linear combination of every $\revone{\hat{\bm{w}}}_i$, because the vector $\bm{F}_{k+1,0}$ is always transverse to the boundary at the point $\bdypt_{k+1}$, by assumption.  Therefore, the vector $\bm{F}_{k+1,0}$ and the $n-1$ orthonormal basis vectors of the tangent hyperplane at the point $\bdypt_{k+1}$ are 
linearly independent, and the matrix $\yp{C}_{k+1}$ is invertible.

The boundary crossing point considered in the proof is arbitrary, so the proof applies to all boundary crossings of the limit cycle.
%, and .  Therefore, Theorem \ref{theorem} may be used to calculate all discontinuities in the iPRC for differential inclusions satisfying hypotheses \ref{h1}--\ref{h4} and assumptions \ref{a:limit_cycle}--\ref{a:direct_deriv}, and in the case of piecewise linear and discontinuous vector fields, derive explicit expressions for the discontinuous iPRC.
\qed

\subsection{Proof of Corollary \ref{corl:circuit}}
\label{section:proof-of-corollaryIII2}
We adopt the same assumptions as Theorem \ref{theorem}.  By the assumptions stated in Section \ref{section:definitions-and-hypotheses}, the limit cycle, $\revone{\bm{\gamma}}$, consists of $K$ distinct sections, each passing through a linear vector field $\bm{F}_1,\bm{F}_2,\ldots,\bm{F}_K$. 
The solution to the adjoint equation of the first region  (Eq.~\eqref{eq:adjoint-piecewise}) is 
\begin{equation}
\bm{z}_1(t) = e^{A_{1} t} \hat{\bm{z}}_{1,0},
\end{equation}
where $A_{1} = -DF_1(\revone{\bm{\gamma}}_1(t))^T$, the negative transpose of the Jacobian matrix $DF_1$ evaluated along the limit cycle $\revone{\bm{\gamma}}(t)$, $e^{A_{1} t}$ is a matrix exponential, and $\hat{\bm{z}}_{1,0}$ is an initial condition of the iPRC in the first region.  Because the vector field is linear within each region, the Jacobian matrix is piecewise constant.  The initial condition of the iPRC of the next region $\bm{z}_{2,0}$ may be written in terms of the initial condition of the iPRC of the first region:
\begin{equation}
\begin{split}
 \bm{z}_{2,0} &= M_{2} \bm{z}_{1,t_1}\\
 &= M_{2} e^{A_{1} t_1} \hat{\bm{z}}_{1,0}.
 \end{split}
\end{equation}
Similarly, the initial condition of the iPRC of the third region may be written as,
\begin{equation}
 \begin{split}
  \bm{z}_{3,0} &= M_{3} \bm{z}_{2,t_2}\\
  &= M_{3} e^{A_{2} t_2} M_{2} e^{A_{1} t_1} \hat{\bm{z}}_{1,0},
 \end{split}
\end{equation}
and so forth.  Upon traversing the $K$th region we return to $ \hat{\bm{z}}_{1,0}$, which must satisfy
\begin{equation}
 \hat{\bm{z}}_{1,0}=M_{1}e^{A_{K}t_{K}}\cdots M_{2}e^{A_{1} t_{1}} \hat{\bm{z}}_{1,0} =: B \hat{\bm{z}}_{1,0}.
\end{equation}
Therefore $\hat{\bm{z}}_{1,0}$ is a unit eigenvector of $B$.  Uniqueness of the unit eigenvector (up to multiplication by a constant) follows from the stability of the limit cycle (assumption \ref{a:limit_cycle}).  If $B\mathbf{q}=\mathbf{q}$ for another eigenvector $\mathbf{q}\not\in\rm{span}(\mathbf{\hat{\mathbf{z}}}_{1,0})$, an arbitrarily small initial condition could be found near $\bdypt_{1}$ that did not converge to $\revone{\bm{\gamma}}$.  

% Uniqueness of this eigenvector is guaranteed in the following way.  Because $\hat{\bm{z}}_{1,0}$ and $1$ are an eigenvector and eigenvalue pair, it follows that
%\begin{equation}
%(B-I) \hat{\bm{z}}_{1,0} = 0.
%\end{equation}
%In other words, $\hat{\bm{z}}_{1,0}$ belongs to the nullspace of the matrix $(B-I)$. If the nullspace of $(B-I)$ has dimension greater than one, then there is a vector $q$ that is not a scalar multiple of $\hat{\bm{z}}_{1,0}$ such that
%\begin{equation}
%(B-I)q = 0,
%\end{equation}
%so $q$ is another eigenvector of the matrix $B$ with unit eigenvalue, and $\hat{\bm{z}}_{1,0}$ is not unique.  Therefore, requiring that the dimension of the nullspace of $(B-I)$ be no greater than one guarantees uniqueness of the eigenvector associated with unit eigenvalue. If $\hat{\bm{z}}_{1,0}$ is an eigenvector of $B$, then the direction of the initial condition of the iPRC, $\hat{\bm{z}}_{1,0}$, is guaranteed to be unique by Corollary \ref{corl:circuit}. 
 
 Uniqueness of the magnitude of $\hat{\bm{z}}_{1,0}$ comes from Eq.~\eqref{eq:normalization}, which we recall to be,
\begin{equation}\label{eq:normalization-proof}
\bm{F}_{1,0}\cdot  \bm{z}_{1,0}=\frac{1}{T},
\end{equation}
where $\bm{z}_{1,0}$ is the unique initial condition of the iPRC of region 1.  The vector $\hat{\bm{z}}_{1,0}$ must be scaled by some constant $\alpha$ to be equivalent to the initial condition, $\bm{z}_{1,0}$.  We can calculate the scaling by use of Eq.~\eqref{eq:normalization-proof}:
\begin{equation}
\begin{split}
\bm{F}_{1,0} \cdot (\alpha \hat{\bm{z}}_{1,0}) &=\frac{1}{T}\\
 \alpha \bm{F}_{1,0} \cdot \hat{\bm{z}}_{1,0} &=\frac{1}{T}\\
 \alpha &=\frac{1}{T ( \bm{F}_{1,0}\cdot \hat{\bm{z}}_{1,0} )}.
 \end{split}
\end{equation}
Thus the unique initial condition of the iPRC, $\bm{z}_{1,0}$ is
\begin{equation}
\begin{split}
 \bm{z}_{1,0} &= \alpha \hat{\bm{z}}_{1,0}\\
 &= \frac{\hat{\bm{z}}_{1,0}}{T ( \bm{F}_{1,0}\cdot \hat{\bm{z}}_{1,0} ) }.
 \end{split}
\end{equation}
This concludes the proof of uniqueness for the initial condition of the iPRC for affine linear vector fields satisfying Theorem \ref{theorem}. \qed

\subsection{Proof of Corollary \ref{corl:eqn}}
\label{section:proof-of-corollaryIII3}
\ymp{Let $t$ denote global time, $T_k = \sum_{j=1}^{k} t_j$ where $t_k$ is the time of flight for region $k$, ($1\le k \le K$), and $A_k = -(DF_k)^T$, where $DF_k$ is the Jacobian matrix of the vector field $\bm{F}_k$. Beginning with region 1, the affine linear vector field assumption allows us to explicitly solve for the iPRC using matrix exponentials. We determine the initial condition of region 1, $\bm{z}_{1,0}$, using Corollary \ref{corl:circuit}. Then for any $0 \leq t < T_1$,
\begin{equation}
 \bm{z}(t) = e^{A_1 t}\bm{z}_{1,0}.
\end{equation}
This iPRC solution reaches the next boundary at time $t_1$, and takes the value $\bm{z}(t_1^-) = e^{A_1 t_1}\bm{z}_{1,0}$. To continue the solution into the next region, we compute the jump matrix $M_{2}$ using Theorem \ref{theorem} and apply it to the value $\bm{z}(t_1^-)$ to place the iPRC solution on the  initial condition on the boundary of region 2, $\bm{z}_{2,0} \equiv M_{2} e^{A_1 t_1}\bm{z}_{1,0}$. Using this initial condition, and for any $T_1 < t < T_2$, we compute the iPRC solution to be
\begin{equation}
 \bm{z}(t) = e^{A_2 (t-T_1)} z_{2,0} \equiv e^{A_2 (t-T_1)}  M_{2} e^{A_1 t_1}\bm{z}_{1,0}.
\end{equation}
By repeating this process into the final region $K$, the solution results in an alternating product of jump matrices and matrix exponentials for any $T_{K-1} < t < T_K$:
\begin{equation}
 \bm{z}(t) = e^{A_k (t-T_{K-1})}M_{K} e^{A_{K-1} t_{K-1}} \cdots  M_{2} e^{A_1 t_1}\bm{z}_{1,0}.
\end{equation}
This equation completes the proof.
\qed

\bl{\begin{remark}Because the iPRC solution is periodic and isolated, continuing this process from region $K$ into region 1 results in the solution
\begin{equation}
\begin{split}
 \bm{z}(t) &= e^{A_1 t} M_{K}e^{A_k (t_K)} \cdots M_{2} e^{A_1 t_1}\bm{z}_{1,0}\\
 &= e^{A_1 t} B \bm{z}_{1,0},
\end{split}
\end{equation}
 for  $nT \leq t < n T+T_1$, for any integer $n$.
The last line is a trivial substitution that follows from Corollary \ref{corl:circuit}.
\end{remark}
}}

\subsection{Proof of Corollary \ref{corl:c0-vector-fields}}
\label{section:proof-of-corollaryIII4}
If the vector fields are of class $C^0$ over boundaries, then the vector field coordinates $a_i$ and $b_i$ (see Eq.~\eqref{eq:matrix-a-b}) are equal for each $i$.  Therefore, the matrices $\yp{C}_{k+1}$ and $\yp{D}_k$ of Eq.~\eqref{eq:matrix-a-b} are identical, and because $M_{k+1}=\yp{C}_{k+1}^{-1}\yp{D}_k$, the matrix $M_{k+1}$ reduces to the identity matrix for each $k$. \qed

\section{\yp{Calculation Details for the Nominal Biting Model of \textit{Aplysia Califonica}}}\label{a:aplysia_details}
Recall Equation \eqref{eq:piecewise-linear-heteroclinic-full}:
\begin{equation*}
\frac{d\bm{r}}{dt} = \left \{ \begin{array}{cl}
  \begin{matrix}1-x -(y+a) \rho \\ y + a \\ (z-a)(1-\rho)\end{matrix},& \quad x \geq y+a, x \geq z-a,\\
  \vspace{5pt}\\
  \begin{matrix}(x-a)(1-\rho) \\ 1-y-(z+a) \rho  \\ z + a\end{matrix},& \quad  y > x-a, y \geq z+a,\\
  \vspace{5pt}\\
  \begin{matrix}x+a\\ (y-a)(1-\rho)\\1 -z - (x+a)\rho \end{matrix},& \quad z > x+a, z > y-a, \\
\end{array}\right .
\end{equation*}

In the domain, the three corners given by the standard basis vectors each have a saddle point Because the vector field is linear within each region, the stable manifold is a plane spanned by the two stable eigenvectors of the Jacobian for each region, and the unstable manifold is the (half) line in the direction of the unstable eigenvector. \ymp{The two stable eigenvectors and the unstable eigenvector for region 1 are, respectively,
\begin{equation}\label{eq:biting_eigenvectors1}
 (1, 0, 0), (0, 0, 1),\left (-\frac{\rho}{2}, 1, 0 \right ).
\end{equation}
The vectors for the saddle in region 2 are, respectively,
\begin{equation}
 (0, 1, 0), (1, 0, 0), \left (0, -\frac{\rho}{2}, 1 \right),
\end{equation}
and for region 3,
\begin{equation}\label{eq:biting_eigenvectors3}
 (0, 0, 1), (0, 1, 0), \left (1, 0, -\frac{\rho}{2}\right).
\end{equation}}

When $a > 0$ the heteroclinic cycle is broken, and the unstable manifold of each vector field $\bm{F}_k$ flows into the boundary surface between vector fields $\bm{F}_k$ and $\bm{F}_{k+1}$ (as opposed to flowing along the boundary edge when $a=0$, where there is a nonempty intersection of the unstable manifold of $\bm{F}_k$ and the stable manifold of $\bm{F}_{k+1}$).  The vector fields $\bm{F}_1$, $\bm{F}_2$, and $\bm{F}_3$, are shifted by vectors $s_1$, $s_2$, and $s_3$, respectively, where
\begin{equation}
\begin{split}
 s_1 &= (0,-a,a),\\
 s_2 &= (a,0,-a), \\
 s_3 &= (-a,a,0).
\end{split}
\end{equation}

%When the system is modified in this way, the heteroclinic cycle breaks and gives rise to limit cycles.

As in the example of the modified iris system, the limit cycle of this nominal piecewise linear SHC model is not analytically computable. The time of flight for each portion of the limit cycle, $t_k$, must be derived numerically for each $k$ and for fixed parameter values $a$ and $\rho$. The limit cycle coordinates are obtained numerically; we denote them  $\bdypt_{k+1} = (\eta_{k+1},\kappa_{k+1},\nu_{k+1})$, i.e., for the $k$th portion of the limit cycle, its initial value is the vector $\bm{\gamma}_{k,0} = (\eta_{k},\kappa_{k},\nu_{k})$.  We now calculate the jump matrices $M_{k+1}$ for each boundary $\bdypt_{k+1}$, beginning with region 1.

The normal vector at $\bdypt_{2}$ and its two tangent vectors, $\hat{\bm{u}}_{2}$ and $\hat{\bm{w}}_{2}$, are
\begin{equation}
\begin{split}
\n_{2} &=\left (-\frac{1}{\sqrt{2}},\frac{1}{\sqrt{2}},0 \right ),\\
\hat{\bm{u}}_{2} &= \left( 0,0,1 \right ), \\
\hat{\bm{w}}_{2} &= \left (-\frac{1}{\sqrt{2}},-\frac{1}{\sqrt{2}},0 \right ).
\end{split}
\end{equation}
By Theorem \ref{theorem}, the jump matrix $M_{2}$is
%
% \begin{equation}\label{eq:shc-m1}
% \begin{split}
% M_{2} = &\left( \begin{matrix}
% \left [(1-\rho ) \left(\eta _2-a\right) \right] & \left[ 1 -\rho  \left(a+\nu _2\right)-\kappa_2 \right] & \left[ a+\nu_2\right ]\\
% 0 & 0 & 1\\
% -\frac{1}{\sqrt{2}} & -\frac{1}{\sqrt{2}} & 0
% \end{matrix} \right )^{-1}\\
% &\left( \begin{matrix}
% \left[ 1-\rho  \left(a+\kappa _2\right)-\eta_2\right] & \left[a+\kappa_2\right] & \left[(1-\rho ) \left(\nu_2-a\right) \right]\\
% 0 & 0 & 1\\
% -\frac{1}{\sqrt{2}} & -\frac{1}{\sqrt{2}} & 0
% \end{matrix}
% \right )\\
% = &\frac{1}{\varsigma_1}\left(
% \begin{array}{ccc}
%  \eta _2+\psi_1  \kappa _2- \phi_1 & 1-a(\rho  +1)-2 \kappa _2-\phi_1 & \omega_1 \\
%  \mu_1+(\rho -2) \eta _2-\rho  \kappa _2+1 & \kappa_2-a (\rho -2)+\eta_2\psi_1 & -\omega_1 \\
%  0 & 0 & \mu_1 +\eta_2 \psi_1 -\kappa _2-\phi_1+1 \\
% \end{array}
% \right),
% \end{split}
% \end{equation}
% %
% where $\mu_1 = a(1-2 \rho)$, $\phi_1 = \rho \nu_2$, $\psi_1 = \rho-1$, $\varsigma_1 = \left( 1-2 a \rho +a+\eta _2 (\rho -1)-\kappa _2-\nu _2 \rho  \right)$ and $\omega_1 = \left(\rho 
%    \nu _2-a (\rho -2) \right )$.  The final two matrices, $M_{3}$ and $M_{1}$, are,
\begin{equation}\label{eq:shc-m1}
\begin{split}
M_{2} = \left(
\begin{array}{ccc}
 a_{11} & a_{12} & a_{13}\\
 a_{21} & a_{22} & -a_{13}\\
 0 & 0 & a_{33}
\end{array}
\right),
\end{split}
\end{equation}
where
\begin{equation}
\begin{split}
 a_{11} &= \eta_2 + (\rho-1)\kappa_2 - \rho \nu_2,\\
 a_{12} &= 1 - a(1+\rho) -2\kappa_2 - \rho \nu_2,\\
 a_{13} &= -a(\rho-2) + \rho \nu_2,\\
 a_{21} &= 1 + a(1 - 2\rho) + (\rho-2)\eta_2 - \rho \kappa_2,\\
 a_{22} &= \kappa_2 - a(\rho-2) + (\rho-1)\eta_2,\\
 a_{33} &= 1 + a(1-2\rho) + (\rho-1)\eta_2 - \kappa_2 - \rho
\end{split}
\end{equation}
The remaining jump matrices are
\begin{equation}\label{eq:shc-m2}
\begin{split}
M_{3} = \left(
\begin{array}{ccc}
 b_{11} & 0 & 0\\
 b_{21} & b_{22} & b_{23}\\
 -b_{21} & b_{32} & b_{33}
\end{array}
\right),
\end{split}
\end{equation}
where
\begin{equation}
\begin{split}
 b_{11} &= a(2\rho-1)+\rho  \eta _3-(\rho -1) \kappa _3+\nu _3-1, \\
 b_{21} &= a (\rho -2)-\rho  \eta _3,\\
 b_{22} &= \rho  \eta _3-\kappa _3-(\rho -1) \nu _3,\\
 b_{23} &= a(\rho+1)+\rho  \eta _3+2 \nu _3-1,\\
 b_{32} &= a(2 \rho-1)-(\rho -2) \kappa _3+\rho  \nu _3-1,\\
 b_{33} &= a (\rho -2)-(\rho -1) \kappa _3-\nu _3,
\end{split}
\end{equation}
and
\begin{equation}\label{eq:shc-m3}
\begin{split}
M_{1} = \left(
\begin{array}{ccc}
 c_{11} & c_{12} & c_{13}\\
 0 & c_{22} & 0\\
 c_{31} & -c_{12} & c_{33}
\end{array}
\right),
\end{split}
\end{equation}
where
\begin{equation}
\begin{split}
 c_{11} &= a (\rho -2)-\eta _1-(\rho -1) \nu _1,\\
 c_{12} &= \rho  \kappa _1-a (\rho -2),\\
 c_{13} &= 2 \rho  a-a+\rho \eta _1-(\rho -2) \nu _1-1,\\
 c_{22} &= 2 \rho  a-a+\eta _1+\rho  \kappa _1-\rho  \nu _1+\nu _1-1,\\
 c_{31} &= \rho  a+a+2 \eta _1+\rho  \kappa _1-1,\\
 c_{33} &= -(\rho -1) \eta_1+\rho  \kappa _1-\nu _1.
\end{split}
\end{equation}

% %
% \begin{equation}\label{eq:shc-m2}
% M_{3} =\frac{1}{\varsigma_2} \left(
% \begin{array}{ccc}
%  2 \rho  a-a+\rho  \eta _3-(\rho -1) \kappa _3+\nu _3-1 & 0 & 0 \\
%  \omega_2 & \rho  \eta _3-\kappa _3-(\rho -1) \nu _3 & \rho  a+a+\rho 
%    \eta _3+2 \nu _3-1 \\
%  -\omega_2 & 2 \rho  a-a-(\rho -2) \kappa _3+\rho  \nu _3-1 & a (\rho
%    -2)-(\rho -1) \kappa _3-\nu _3 \\
% \end{array} \right)
% \end{equation}
% %
% where $\varsigma_2 = \left ( a (2 \rho -1)+\eta _3 \rho -\kappa _3 \rho +\kappa _3+\nu _3-1 \right ) $ and $\omega_2 = \left (a (\rho -2)-\rho  \eta _3 \right )$, and
% %
% \begin{equation}\label{eq:shc-m3}
% \begin{split}
%   M_{1}= \frac{1}{\varsigma_3}\left(
% \begin{array}{ccc}
%  a (\rho -2)-\eta _1-(\rho -1) \nu _1 & \omega_3 & 2 \rho  a-a+\rho 
%    \eta _1-(\rho -2) \nu _1-1 \\
%  0 & 2 \rho  a-a+\eta _1+\rho  \kappa _1-\rho  \nu _1+\nu _1-1 & 0 \\
%  \rho  a+a+2 \eta _1+\rho  \kappa _1-1 & -\omega_3 & -(\rho -1) \eta
%    _1+\rho  \kappa _1-\nu _1 \\
% \end{array}
% \right),
% \end{split}
% \end{equation}
%
 % \todo{see aplysia /papers /yxp30 /j-math-neuro /nominal \_biting \_adjoint.nb for calculation of the jump matrices}
%where $\varsigma_3 = (2 a \rho -a+\eta _1+\kappa _1 \rho -\nu _1 \rho +\nu _1-1)$ and $\omega_3 = (\rho  \kappa _1-a (\rho -2))$.
The solution to the adjoint equation for each region is
\begin{equation}\label{eq:shc-zi}
\begin{split}
\bm{z}_1(t) &= \left ( \begin{matrix} 
		    e^t & 0 & 0\\
		    \rho\sinh(t)  & e^{-t} & 0\\
		    0 & 0 & e^{t(\rho-1)}
		  \end{matrix} \right ) \bm{z}_{1,0},\\
\bm{z}_2(t) &= \left ( \begin{matrix}
		  e^{t(1- \rho)} & 0 & 0\\
		  0 & e^t & 0\\
		  0 & \rho\sinh(t)  & e^{-t}
		  \end{matrix} \right ) \bm{z}_{2,0},\\
\bm{z}_3(t) &= \left ( \begin{matrix}
		  e^{-t} & 0 & \rho\sinh(t) \\
		  0 & e^{t(\rho-1)} & 0 \\
		  0 & 0 & e^t
		  \end{matrix} \right ) \bm{z}_{3,0}.
\end{split}
\end{equation}
% 
% \section{Derivation of the iPRC for Smooth Systems}\label{iprc_derivation-appendix}
% The change in phase, $\Delta \theta$, in response to a perturbation of size $\varepsilon$ in the unit vector direction $\eta$, depends on the time at which the perturbation occurs.  To derive the infinitesimal phase response curve, when the asymptotic phase function $\theta$ is $C^1$, we expand the function to first order in $\varepsilon$.
% \begin{equation}\label{eq:iprc_derivation}
% \theta(\revone{\bm{\gamma}}(t) + \varepsilon \eta) = \theta(\revone{\bm{\gamma}}(t)) + \varepsilon D\theta(\revone{\bm{\gamma}}(t))\cdot \eta + O(\varepsilon^2)
% \end{equation}
% where $D$ denotes the directional derivative.
% %
% The change in phase for a small perturbation of magnitude $\varepsilon$ and direction $\eta$ is 
% %
% \begin{equation}
%  \Delta \theta = \theta(\revone{\bm{\gamma}}(t) + \varepsilon \eta) - \theta(\revone{\bm{\gamma}}(t)) = \varepsilon D\theta(\revone{\bm{\gamma}}(t))\cdot \eta + O(\varepsilon^2)
% \end{equation}
% %
% Taking the limit results in the iPRC
% %
% \begin{equation}
% \lim_{\varepsilon \rightarrow 0} \frac{\Delta \theta}{\varepsilon} = D\theta(\revone{\bm{\gamma}}(t))\cdot \eta =: \revone{\bm{z}}(t) \cdot \eta.
% \end{equation}
% %:
% \bl{For a classical derivation of the adjoint equation see \cite{BrownMoehlisHolmes2004NeComp,ErmentroutTerman2010book,SchwemmerLewis2012PRCchapter}.}

\section{Acknowledgments} This work was supported in part by NSF grant DMS-1413770 and NSF grant DMS-1010434.  {The authors thank Dr.~Yangyang Wang for suggesting several improvements to the manuscript.}
\bibliographystyle{dcu}
\bibliography{Aplysia,neuroscience,math,PJT,physics,reliability}

\end{document}